\documentclass[arxiv,reqno,twoside,a4paper,12pt]{amsart}

\usepackage{enumerate,amssymb,amsmath,mathrsfs}
\usepackage[colorlinks=true,linkcolor=blue,citecolor=blue]{hyperref}
\usepackage{amssymb,amsfonts,amsthm,amscd}
\usepackage[mathscr]{eucal}     
\usepackage{graphicx}
\usepackage[arrow, matrix, curve]{xy}
\usepackage{color}
\usepackage[margin=3cm]{geometry}
\definecolor{gray}{gray}{.75}
\definecolor{gray2}{gray}{.50}
\usepackage{tikz}

\usepackage{tikz-cd, pgfplots}

\usepackage[scaled]{helvet} 
\usepackage{courier} 
\usepackage[mathbf]{euler}

\theoremstyle{plain}
\newtheorem{thm}{Theorem}[section]
\newtheorem{prop}[thm]{Proposition}
\newtheorem{cor}[thm]{Corollary}
\newtheorem{lem}[thm]{Lemma}
\newtheorem{lemma}[thm]{Lemma}
\newtheorem{remark}[thm]{Remark}
\newtheorem{defn}[thm]{Definition}

\newtheorem{assump}[thm]{Assumption}

\newcommand{\R}{\mathbb{R}}
\newcommand{\C}{\mathbb{C}}

\newcommand{\supp}{\mathrm{supp}}
\newcommand{\dom}{\mathscr{D}}

\usepackage{xcolor}
\definecolor{darkgreen}{cmyk}{1,0,1,.2}
\definecolor{m}{rgb}{1,0.1,1}

\numberwithin{equation}{section}

\definecolor{qqwuqq}{rgb}{0,0,0}

\usepackage{color}

\begin{document}

\date{\today}
\title[Stability of $L^2-$invariants on stratified spaces]
{Stability of $L^2-$invariants on stratified spaces}

\author{Francesco Bei}
\address{Sapienza University, Rome, Italy} 
\email{francesco.bei@uniroma1.it}

\author{Paolo Piazza}
\address{Sapienza University, Rome, Italy} 
\email{paolo.piazza@uniroma1.it} 

\author{Boris Vertman} 
\address{Universit\"at Oldenburg, Germany} 
\email{boris.vertman@uni-oldenburg.de}


\begin{abstract}
Let $\overline{M}$ be a compact smoothly stratified pseudo-manifold endowed 
with a wedge metric $g$. Let $\overline{M}_\Gamma$ be a Galois $\Gamma$-covering. 
Under additional assumptions on 
$\overline{M}$, satisfied for example by Witt pseudo-manifolds, we show that
the  $L^2$-Betti numbers and the Novikov-Shubin invariants are well defined. 
We then establish their invariance
under a smoothly stratified codimension-preserving homotopy equivalence,
thus extending results of Dodziuk, Gromov and Shubin  to these
pseudo-manifolds.
\end{abstract} 

\maketitle \ \\[-18mm]

\tableofcontents

\section{Introduction and statement of the main results}

\subsection{Historical overview}\label{history-subsection}

Let $(M,g)$ be a compact manifold with an infinite fundamental group. 
In his seminal paper \cite{Ati} Atiyah introduced $L^2$-Betti numbers associated 
to the universal covering of $M$, and conjectured their stability 
under homotopy equivalences. The conjecture was established  shorty after the appearance
of Atiyah's paper by Dodziuk in \cite{Dodziuk}.
Later, Novikov and Shubin \cite{Novikov-Shubin1, Novikov-Shubin2} introduced
new invariants associated to a Galois $\Gamma$-covering $M_\Gamma$ of $M$, with $\Gamma$ a finitely 
generated discrete group; these invariants 
measure the density of the continuous spectrum of the differential form Laplacian
on $M_\Gamma$ and are nowadays referred to as the
Novikov-Shubin invariants. Their stability under homotopy equivalence was proved by 
Gromov and Shubin in \cite{Gromov-Shubin, Gromov-Shubin-err}. 
\medskip

{\it The goal of this article
is to  address the same stability problems but for singular spaces, namely 
compact smoothly stratified (Thom-Mather) pseudo-manifolds.}

\subsection{Statement of the main results}

Consider a compact smoothly stratified (Thom-Mather) pseudo-manifold $\overline{M}$
with an iterated wedge metric $g$ in  its open interior $M \subset \overline{M}$.
Simply put, the singular neighborhoods of $\overline{M}$ can locally be viewed as fibrations of 
cones, as in Figure \ref{figure1}, where their cross sections may be singular again.
\medskip

\begin{figure}[h]
	\includegraphics[scale=0.5]{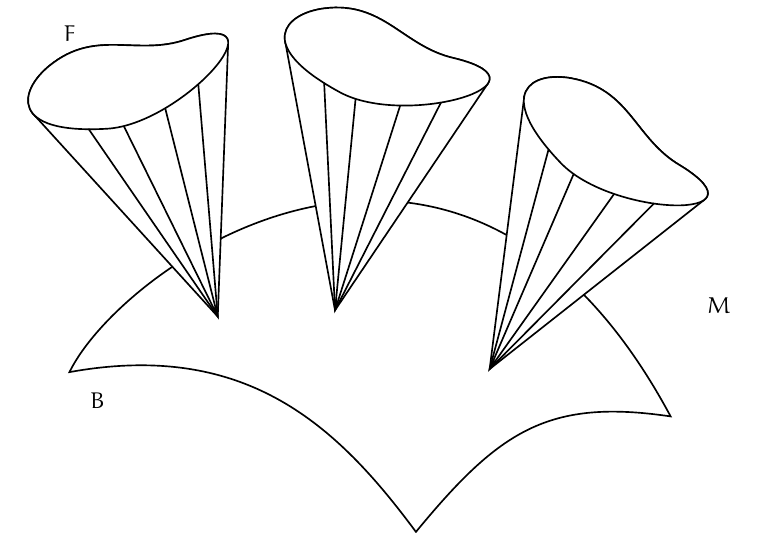}
	\caption{A cone bundle over $B$.}
	\label{figure1}
\end{figure}

Let $p: \overline{M}_\Gamma \to \overline{M}$ be a Galois $\Gamma$-covering
with $\Gamma$ being the group of deck transformations $\Gamma$. It is again a stratified pseudo-manifold and 
the metric $g$ lifts to an iterated wedge metric $g_\Gamma$ in the 
interior $M_\Gamma$ of $\overline{M}_\Gamma$. The corresponding minimal and maximal
Hilbert complexes $(\dom_{\min}(M_\Gamma), d_\Gamma)$ and $(\dom_{\max}(M_\Gamma), d_\Gamma)$ arise from the de Rham complex 
of compactly supported differential forms on $M_\Gamma$ by 
choosing minimal and maximal closed extensions of the differential in $L^2$, respectively. Note that in general 
$$
(\dom_{\min}(M_\Gamma), d_\Gamma) \neq (\dom_{\max}(M_\Gamma), d_\Gamma).
$$
The $L^2$-Betti numbers are defined as von Neumann $\Gamma$-dimensions of the corresponding reduced $L^2$-cohomologies 
\begin{equation}
\begin{split}
b^*_{(2),\min}(\overline{M}_\Gamma) := \dim_\Gamma \overline{H}^*_{(2)}(\dom_{\min}(M_\Gamma), d_\Gamma), \\
b^*_{(2),\max}(\overline{M}_\Gamma) := \dim_\Gamma \overline{H}^*_{(2)}(\dom_{\max}(M_\Gamma), d_\Gamma).
\end{split}
\end{equation}
The Novikov-Shubin invariants $\alpha_{*,\min}(\overline{M}_\Gamma)$ and $\alpha_{*,\max}(\overline{M}_\Gamma)$ are in turn defined  in terms of the Hodge Laplacians associated to the minimal and maximal Hilbert complexes 
$(\dom_{\min}(M_\Gamma), d)$ and $(\dom_{\max}(M_\Gamma), d)$,
respectively. These invariants will be introduced explicitly in \S \ref{finiteness-section} below. 
We can now state our first main theorem.

\begin{thm}\label{main1} Let $p: \overline{M}_\Gamma \to \overline{M}$ be a Galois $\Gamma$-covering of a compact smoothly stratified
pseudo-manifold $\overline{M}$ endowed with an iterated wedge metric. If the Assumption \ref{fundamental} holds true, we have the following properties:
\begin{enumerate}
\item The $L^2$-Betti numbers $b^*_{(2),\min}(\overline{M}_\Gamma)$
and $b^*_{(2),\max}(\overline{M}_\Gamma)$ are finite. 
\item The Novikov-Shubin invariants 
$\alpha_{*,\min}(\overline{M}_\Gamma)$ and $\alpha_{*,\max}(\overline{M}_\Gamma)$ are well-defined.
\end{enumerate}
\end{thm}

\noindent
Assumption \ref{fundamental} is a technical assumption
on the resolvent $(i+D_{{\rm abs/rel}})^{-1}$ on $(M,g)$,
with $D_{{\rm abs/rel}}$ being the absolute and relative self-adjoint extensions
$$
D_{\textup{rel}}:= d_{\min} + d^*_{\min}, \quad \text{and}
\quad D_{\textup{abs}}:= d_{\max} + d^*_{\max}\,.
$$
It is satisfied on stratified pseudo-manifolds with isolated singularities and on Witt spaces of arbitrary depth.,
as shown in Proposition \ref{witt-examples} below.

In order to state our second main result 
we pause for a moment and explain in detail the result of Gromov and Shubin \cite{Gromov-Shubin}. Let
$X$ and $Y$ be connected smooth compact manifolds without boundary and let $f:X\to Y$ be a homotopy equivalence.
We now consider two connected Galois $\Gamma$-coverings $X_\Gamma\xrightarrow{p} X$ and $Y_\Gamma\xrightarrow{q} 
Y$; once we fix 
base points $x\in X, \tilde{x}\in p^{-1}(x)$, $f(x)\in Y$ and $\tilde{y}\in q^{-1}(f(x))$ we have  well
defined surjective homomorphisms
$$\pi_1 (X,x)\xrightarrow{j_X} \Gamma\,,\quad \pi_1 (Y,f(x))\xrightarrow{j_Y} \Gamma.$$
We shall explain all this later in the paper.
Gromov and Shubin requires the following diagram to commute:
\begin{figure}[h!]
\begin{center}
\begin{tikzpicture}

\draw[->] (5.6,0) -- (7,0);
\node at (4.5,0) {$\pi_1 (Y,f(x))$};
\node at (7.5,0) {$\Gamma$};
\node at (6.3,-0.5) {$j_Y$};

\draw[->] (5.4,2) -- (6.8,2);
\node at (4.5,2) {$\pi_1 (X,x)$};
\node at (7.5,2) {$\Gamma$};
\node at (6,2.5) {$j_X$};

\draw[->] (4.5,1.5) -- (4.5,0.5);
\draw[->] (7.5,1.5) -- (7.5,0.5);

\node at (4,1) {$f_*$};
\node at (7.9,1) {${\rm id}$};

\end{tikzpicture}
\end{center}
\end{figure}

We shall weaken this assumption of Gromov and Shubin, that is $j_X=j_Y\circ f_*$,  and require only that
\begin{equation}\label{comp}
f_* ({\rm Ker} (j_X))={\rm Ker} (j_Y)\,.
\end{equation}
This is  condition \eqref{compatibility} appearing in the statement of our next result. In order to state 
it, we  need to introduce one final ingredient $-$ a smoothly stratified codimension-preserving homotopy 
equivalence between compact smoothly stratified pseudo-manifolds. This is a homotopy equivalence that maps singular strata onto singular strata and
preserves co-dimension. We will be more precise in \S \ref{stability-section} below. Our second main result now reads as follows.

\begin{thm}\label{main2} Let $\overline{M}$ and $\overline{N}$ be 
oriented compact smoothly stratified pseudo-manifolds which satisfy 
Assumption \ref{fundamental};
assume that there exist two Galois $\Gamma$-coverings $\overline{M}_{\Gamma}$, $\overline{N}_{\Gamma}$ and
a smoothly stratified codimension-preserving homotopy 
equivalence $f: \overline{M} \to \overline{N}$ satisfying condition \eqref{comp} (see \eqref{compatibility} 
for precise notation). Then
\begin{align*}
b^*_{(2), \max}(\overline{M}_{\Gamma}) &= b^*_{(2),\max}(\overline{N}_{\Gamma}), \quad 
b^*_{(2),\min}(\overline{M}_{\Gamma}) = b^*_{(2),\min}(\overline{N}_{\Gamma}), \\
\alpha_{*, \max}(\overline{M}_{\Gamma}) &= \alpha_{*, \max}(\overline{N}_{\Gamma}), \qquad 
\alpha_{*,\min}(\overline{M}_{\Gamma}) = \alpha_{*,\min}(\overline{N}_{\Gamma}).
\end{align*}
\end{thm}

As a consequence of Corollary \ref{pullback-u}, the condition \eqref{compatibility} is 
always true when we deal with universal coverings. Therefore we have the following result.

\begin{cor}
\label{universal-c}
Let $\overline{M}$ and $\overline{N}$ be two oriented 
compact smoothly stratified pseudo-manifolds which satisfy 
Assumption \ref{fundamental};
assume that there exists a smoothly stratified codimension-preserving homotopy
equivalence $f: \overline{M} \to \overline{N}$. Let $b^*_{(2),\min / \max}(\overline{M})$, $b^*_{(2),\min / \max}(\overline{N})$, $\alpha_{*,\min / \max}(\overline{M})$ and $\alpha_{*,\min / \max}(\overline{N})$ denote the maximal/minimal $L^2$-Betti numbers and Novikov-Shubin invariants of $\overline{M}$ and $\overline{N}$ computed with respect to the corresponding universal coverings. We then have
\begin{align*}
b^*_{(2),\max}(\overline{M}) &= b^*_{(2), \max}(\overline{N}), \quad 
b^*_{(2),\min }(\overline{M}) = b^*_{(2),\min }(\overline{N}), \\
\alpha_{*,\max}(\overline{M}) &= \alpha_{*,\max}(\overline{N}), \qquad 
\alpha_{*,\min }(\overline{M}) = \alpha_{*,\min }(\overline{N}).
\end{align*}
\end{cor}

As anticipated in \S \ref{history-subsection}, this theorem extends the result of Dodziuk \cite{Dodziuk} for $L^2$-Betti numbers
and of Gromov and Shubin \cite{Gromov-Shubin} for Novikov-Shubin invariants
to the setting of smoothly stratified Thom-Mather pseudo-manifolds. 

\begin{remark}
For smooth compact manifolds it is possible to give a topological definition of 
 the $L^2$-Betti numbers and of the Novikov-Shubin invariants. The proof of the equivalence of these two definitions 
 is a non-trivial result, treated in great detail in \cite{Lueck-book}. We believe that it should be possible to employ
 $L^2$-intersection chains \`a la  Goresky-MacPherson on the universal covering of a Witt space in order to give a topological definition of the
 invariants treated in this article. While the definition does not pose particular problems, the proof of the equivalence between the analytic definition given here and this topological definition appears to be rather involved.
We plan to look into these questions in a subsequent paper. 
\end{remark}

\subsection{Background and notation}

This paper builds strongly upon the classical notions of stratified spaces, their Galois coverings and the
corresponding von Neumann theory. A careful introduction of these concepts is out of scope of the present note
and is presented in full detail in other references. We list the main objects of interest
for us, their notation and precise references where these objects are introduced. 

\begin{enumerate}
\item \textbf{stratified space $\overline{M}$, its open interior $M$ and resolution $\widetilde{M}$.}\medskip

\noindent A compact smoothly stratified (Thom-Mather) pseudo-manifold $\overline{M}$ with open 
interior $M \subset \overline{M}$ is defined, for example, in \cite[Definition 2.1]{ALMP3}. The precise definition is rather 
involved, due to additional Thom-Mather conditions, see \cite{Mather}, which guarantee that such  $\overline{M}$ 
can be resolved into a compact manifold with corners and boundary fibration structure, see \cite[Proposition 2.5]{package} and \cite[Definition 2]{package}. 
Following \cite{package} we denote this resolution by $\widetilde{M}$.
Further references are  e.g. \cite{Brasselet}, \cite{Verona}, see also \cite{ALMP2,Alb}. \medskip

\item \textbf{iterated incomplete wedge metric $g$ and complete edge metric $\rho_M^{-2}g$.}\medskip

\noindent An (iterated incomplete) wedge metric $g$ on the open interior $M$ is defined in 
\cite[Definition 5]{package} (note that there such a metric is called 'iterated edge'). If $\rho_M$
is a total boundary function on $\widetilde{M}$, i.e. a smooth function that vanishes to first order at all boundary faces, 
then a complete edge metric is defined by $\rho_M^{-2}g$ as in \cite[(3.2)]{package}\medskip

\item \textbf{Galois covering $\widetilde{M}_\Gamma$ and von Neumann theory.} \medskip

\noindent A Galois covering $\pi: \overline{M}_{\Gamma} \to \overline{M}$
with Galois group $\Gamma$ is again a smoothly stratified (Thom-Mather) pseudo-manifold with the smooth open and 
dense stratum $M_\Gamma$ being a Galois covering of $M$. Any singular stratum $Y_\Gamma$ of 
$\overline{M}_{\Gamma}$ is a Galois covering of a singular stratum $Y$ in $\overline{M}$ of same depth. The lift $g_\Gamma$ of the wedge metric $g$ defines a $\Gamma$-invariant wedge metric on 
$M_\Gamma$. We refer to e.g. \cite[\S 7.1]{PiVe1} and \cite[\S 2.4]{PiVe2} for more details.
\medskip

\item \textbf{edge and wedge tangent bundles ${}^{e}TM, {}^{e}TM_\Gamma$ and ${}^{w}TM, {}^{w}TM_\Gamma$.} \medskip

\noindent The edge tangent bundle ${}^{e}TM$ is a vector bundle over $\widetilde{M}$, defined in \cite[\S 4.1]{package} (note that there
a different notation ${}^{ie}TM$ is used, with the additional upper script i standing for 'iterated'). 
An efficient definition of the wedge tangent bundle ${}^{w}TM$ is obtained by specifying the
smooth sections (of its dual bundle) as 
\begin{equation*}
C^\infty(\widetilde{M}, {}^{w}T^*M) := \rho_M \, C^\infty(\widetilde{M}, {}^{e}T^*M).
\end{equation*}
One defines ${}^{e}TM_\Gamma, {}^{e}TM_\Gamma$ in the same way. 
Consider the $L^2$-completions 
\begin{equation*}
\begin{split}
&  L^2\Omega^*(M,g) :=L^2(M, \Lambda^* ({}^{w}T^*M),g), \\
& L^2\Omega^*(M_{\Gamma},g_{\Gamma}):=L^2(M_\Gamma, \Lambda^* ({}^{w}T^*M_\Gamma),g_\Gamma).
\end{split}
\end{equation*}
We shall often simply abbreviate the spaces as $L^2\Omega^*(M)$ and
$L^2\Omega^*(M_{\Gamma})$, respectively, without keeping track of the metrics.
\medskip

\item \textbf{Minimal and maximal complexes $(\dom_{\min}(M_\Gamma), d_\Gamma)$ and 
$(\dom_{\max}(M_\Gamma), d_\Gamma)$.} \medskip

\noindent The notion of Hilbert complexes has been carefully introduced in \cite{Hilbert}.
The minimal and maximal Hilbert complexes are defined e.g. in \cite[Lemma 3.1]{Hilbert}. They are
defined by the minimal and maximal domains $\dom_{\min / \max}(M_\Gamma):= \dom(d_{\Gamma, \, \min / \max})$ 
in $L^2\Omega^*(M_{\Gamma},g_{\Gamma})$ of de Rham differentials on $M_\Gamma$. 
\end{enumerate}

\noindent \emph{Acknowledgements.} The authors would like to acknowledge 
interesting discussions with Pierre Albin and Matthias Lesch. The third author thanks 
Sapienza University for hospitality. This work was supported by the "National Group for 
the Algebraic and Geometric Structures
and their Applications" (GNSAGA-INDAM).
We thank the referee for useful comments and questions that
certainly improved this article.

\section{Reformulating the problem using Mishchenko bundles}\label{twist-section}

Below, it will become necessary to use an equivalent description of the 
Hilbert complexes $(\dom_{\min}(M_\Gamma), d_\Gamma), (\dom_{\max}(M_\Gamma), d_\Gamma)$,
and the Hilbert space $L^2\Omega^*(M_{\Gamma},g_{\Gamma})$, working only on the base $\overline{M}$
and using the von Neumann algebra framework. We shall be brief here and refer to 
\cite[\S 1 and \S 2]{Lueck-book} and \cite{Schick} for details. 
Consider the complex group ring $\C\Gamma$ of complex-valued maps 
$f: \Gamma \to \C$ with compact support and the Hilbert space of square-summable sequences
$$
\ell^2\Gamma := \Bigl\{f:\Gamma \to \C \ \bigl| \ \sum_{\gamma \in \Gamma} \bigl|f(\gamma)\bigr|^2 < \infty\Bigr\}, 
$$
which can alternatively be viewed as Hilbert space completion of $\C\Gamma$. Recall that the von Neumann
algebra $\mathscr{N}\Gamma$ is defined as the space of bounded $\Gamma$-equivariant linear operators on the Hilbert space $\ell^2\Gamma$.
We recall the notion of $\mathscr{N}\Gamma$-Hilbert modules here,  
see \cite[\S 1 and \S 2]{Lueck-book} and \cite[\S 4]{Gromov-Shubin} for more details

\begin{defn}\label{gamma-notions1} 
A $\mathscr{N}\Gamma$-Hilbert module $H$ is a Hilbert space 
together with a linear isometric $\Gamma$-action such that there exists a Hilbert space $\mathcal{H}$ and an isometric linear 
$\Gamma$-embedding of $H$ into the tensor product of Hilbert spaces $\mathcal{H}\otimes \ell^2 \Gamma$ equipped with the 
obvious $\Gamma$-action. 
\end{defn}

We now define the so-called Mishchenko bundles
\begin{equation}\label{mi-bundles}
\begin{split}
\mathscr{E} &:= \overline{M}_\Gamma\times_\Gamma \ell^2\Gamma \;\longrightarrow \overline{M},\\
\mathscr{E}_c &:= \overline{M}_\Gamma\times_\Gamma \C\Gamma \;\longrightarrow \overline{M}, 
\end{split}
\end{equation}
Fibres of both bundles are modules, the first one for the natural $\mathscr{N}\Gamma$-action
and the second one for the natural $\C\Gamma$ action. In fact, 
the bundle $\mathscr{E}$ is a bundle of $\mathscr{N}\Gamma$-Hilbert modules in the sense of 
\cite[Definition 2.10]{Schick} and the same is true after twisting $\mathscr{E}$  
with 
the vector bundle $\Lambda^* {}^{w}(T^*M)$.
The construction defines natural isomorphisms of vector spaces
\begin{equation}\label{phi-isomorphism}
\begin{split}
\Phi_c: \, &C^\infty_c (M_\Gamma, \Lambda^* ({}^{w}T^*M_\Gamma)) \to C^\infty_c (M, \Lambda^* ({}^{w}T^*M)\otimes \mathscr{E}_c), \\
\Phi' :\, & C^\infty_{(2)} (M_\Gamma, \Lambda^* ({}^{w}T^*M_\Gamma)) \to C^\infty(M, \Lambda^* ({}^{w}T^*M)\otimes \mathscr{E}),
\end{split}
\end{equation}
with 
$C^\infty_{(2)} (M_\Gamma, \Lambda^* ({}^{w}T^*M_\Gamma))$ defined as
$$
\Bigl\{s\in C^\infty (M_\Gamma, \Lambda^* ({}^{w}T^*M_\Gamma)) :
\forall_{x \in M_\Gamma}\sum_\gamma |s(\gamma \cdot x)|^2 < \infty \Bigr\},
$$ 
and $\Phi$, $\Phi'$ written down explicitly in \cite[\S 7.5, (1)]{Schick}.  
The lower index $c$ indicates compact support. 
We shall simplify notation
\begin{align*}
&\Omega^*_{c}(M,\mathscr{E}_c) := C^\infty_c (M, \Lambda^* ({}^{w}T^*M)\otimes \mathscr{E}_c), \quad
&&\Omega^* (M,\mathscr{E}) := C^\infty (M, \Lambda^* ({}^{w}T^*M)\otimes \mathscr{E}), \\
&\Omega^*_{c}(M_\Gamma) := C^\infty_c (M_\Gamma, \Lambda^* ({}^{w}T^*M_\Gamma)).
\end{align*}
We can define a metric $g_{\mathscr{E}}$ on $\Lambda^* ({}^{w}T^*M)\otimes \mathscr{E}$,
using the metric $g$ together with the $\ell^2\Gamma$ inner product along the fibers. Noting in the 
first equality below that $\ell^2\Gamma$ is the Hilbert space completion of $\C\Gamma$, we have
\begin{equation}
\begin{split}
L^2 (M, \Lambda^* ({}^{w}T^*M)\otimes \mathscr{E}, g_{\mathscr{E}}) &= 
\overline{\Omega^*_c  (M, \mathscr{E}_c)}^{\ g_{\mathscr{E}}} \\
L^2 (M_\Gamma, \Lambda^* ({}^{w}T^*M), g_\Gamma) &=
\overline{\Omega^*_c (M_\Gamma)}^{\ g_{\Gamma}}  
\end{split}
\end{equation}
where the closures are in terms of the corresponding $L^2$ inner products
defined by the metrics in the upper indices.
We shall also use the following simplified notation
\begin{equation*}
L^2 \Omega^* (M, \mathscr{E}) := L^2 (M, \Lambda^* ({}^{w}T^*M)\otimes \mathscr{E}, g_{\mathscr{E}}),
\end{equation*}
together with  the one that we have already introduced $$L^2 \Omega^* (M_\Gamma) := L^2(M_\Gamma, \Lambda^* ({}^{w}T^*M_\Gamma),g_\Gamma).$$
In fact, \cite[Lemma 7.10]{Schick} proves that these two spaces are
$\mathscr{N}\Gamma$-Hilbert modules in the sense of \cite[Definition 2.1]{Schick},
and as explained in \cite[\S 7.5, (2)]{Schick}, the isomorphism $\Phi_c$ extends to an isometry  
of $\mathscr{N}\Gamma$-Hilbert modules
\begin{equation}\label{phi-isometry}
\Phi:   L^2 \Omega^*(M_\Gamma)\to L^2 \Omega^* (M, \mathscr{E})
\end{equation}

We now study the Hilbert complexes $(\dom_{\min}(M_\Gamma), d_\Gamma)$ and $(\dom_{\max}(M_\Gamma), d_\Gamma)$
under the isometric identification $\Phi$. There is a canonical connection 
on $\mathscr{E}$ and its subbundle $\mathscr{E}_c$, 
over the interior $M$, induced by the exterior differential on $M_\Gamma$. 
We refer to \cite[\S 3]{Schick} for more details on connections 
in the von Neumann setting. Hence, there is a natural (twisted) de Rham differential $d_{\mathscr{E}}$ on the twisted differential forms 
$\Omega^*_{c}(M,\mathscr{E}_c)$. In fact, $d_{\mathscr{E}}$ and the de Rham differential $d_\Gamma$ on $M_\Gamma$ act on larger 
domains of smooth forms
\begin{equation}\label{smooth-max}
\begin{split}
\Omega^{*}_{\max}(M,\mathscr{E}) &:= \Bigl\{ \omega \in \Omega^{*}(M,\mathscr{E})\cap L^2 \Omega^* (M, \mathscr{E}) 
\mid d_{\mathscr{E}} \omega \in \Omega^{*}(M,\mathscr{E})\cap L^2 \Omega^* (M, \mathscr{E})  \Bigr\}, \\
\Omega^{*}_{\max} (M_\Gamma) &:= \Bigl\{ \omega \in \Omega^* (M_\Gamma)\cap L^2 \Omega^* (M_\Gamma) 
\mid d_\Gamma \omega \in \Omega^* (M_\Gamma)\cap L^2 \Omega^* (M_\Gamma) \Bigr\}.
\end{split}
\end{equation}
Recalling the classical definitions of $(\dom_{\min}(M_\Gamma), d_\Gamma)$ and $(\dom_{\max}(M_\Gamma), d_\Gamma)$, 
we can define the corresponding minimal and maximal Hilbert complexes in $L^2\Omega^*(M, \mathscr{E})$ in the same way. Define
the graph norms 
\begin{equation}
\begin{split}
\| \omega \|_{\mathscr{E}} &:= \| \omega \|_{L^2\Omega^*(M, \mathscr{E})} + \| d_{\mathscr{E}} \omega \|_{L^2\Omega^*(M, \mathscr{E})}, 
\quad \omega \in \Omega^*_{\max}(M,\mathscr{E}), \\
\| \omega \|_\Gamma &:= \| \omega \|_{L^2\Omega^*(M_\Gamma)} + \| d_\Gamma \omega \|_{L^2\Omega^*(M_\Gamma)}, 
\quad \omega \in \Omega^*_{\max}(M_\Gamma).
\end{split}
\end{equation}

\begin{defn}\label{minmax-def} 
Consider the graph norms $\| \cdot \|_{\mathscr{E}}$ and $\| \cdot \|_\Gamma$. 
Then the minimal and maximal Hilbert complexes in $L^2\Omega^*(M, \mathscr{E})$ and 
$L^2\Omega^*(M_\Gamma)$ are defined as follows.
\begin{enumerate}
\item $(\dom_{\min}(M, \mathscr{E}), d_{\mathscr{E}})$ and $(\dom_{\max}(M, \mathscr{E}), d_{\mathscr{E}})$ are defined by
\begin{equation}
\dom_{\min}(M, \mathscr{E}) := \overline{\Omega^*_c(M,\mathscr{E}_c)}^{\,\| \cdot \|_{\mathscr{E}}}, 
\dom_{\max}(M, \mathscr{E}) := \overline{\Omega^*_{ \max}(M,\mathscr{E})}^{\, \| \cdot \|_{\mathscr{E}}}
\end{equation}
\item $(\dom_{\min}(M_\Gamma), d_\Gamma)$ and $(\dom_{\max}(M_\Gamma), d_\Gamma)$ are defined by
\begin{equation}\label{domains-smooth}
\dom_{\min}(M_\Gamma) := \overline{\Omega_c^*(M_\Gamma)}^{\, \| \cdot \|_{\Gamma}}, 
\dom_{\max}(M_\Gamma) := \overline{\Omega^*_{\max}(M_\Gamma)}^{\, \| \cdot \|_{\Gamma}}.
\end{equation}
\end{enumerate}
\end{defn}

\noindent
The definition of the maximal domain in \eqref{domains-smooth} is equivalent to 
the classical distributional definition, see \cite[(2.31)]{Hilbert}. Notice that while the von Neumann theory is
carefully developed in e.g. \cite{Lueck-book} and \cite{Schick}, the definition of the minimal domain
in \eqref{domains-smooth} is only implicit in the literature.

Recall now the notion of $\mathscr{N}\Gamma$-Hilbert complexes and isomorphisms between them:

\begin{defn}\label{gamma-notions2} Let $\Gamma$ be a finitely generated discrete group. 
\begin{enumerate}
\item A Hilbert complex $(\dom(d_*),d_*), \dom(d_*) \subset H_*,$ is an $\mathscr{N}\Gamma$-Hilbert complex,
if each Hilbert space $H_*$ is a $\mathscr{N}\Gamma$-Hilbert module and each differential 
$d_*$ is a closed and densely defined operator commuting 
with the action of $\Gamma$ and satisfying $d_{*+1}\circ d_*\equiv 0$ on the domain of $d_*$. 
\item A morphism between $\mathscr{N}\Gamma$-Hilbert complexes $H_\bullet:=(\dom(d),d)$ with $\dom(d_*) \subset H_*$ and 
$K_\bullet:=(\dom(d'),d')$ with $\dom(d'_*) \subset K_*,$ is defined as a sequence of bounded 
linear operators $f_k:H_k\rightarrow K_k$ for each $k$, commuting with the $\Gamma$ action and 
satisfying $f_{k+1}\circ d_k=d_k\circ f_k$ on $dom(d_k)$. We write 
$f:H_{\bullet}\rightarrow K_{\bullet}$.
\item We say that these $\mathscr{N}\Gamma$-Hilbert complexes are isomorphic, if the maps 
$f_k:H_k\rightarrow K_k$ are isometries and hence $f_k\dom(d_k) = \dom(d'_*)$ for each $k$.
\end{enumerate}
\end{defn}

The minimal and maximal Hilbert complexes above are $\mathscr{N}\Gamma$-Hilbert complexes in the 
sense of Definition \ref{gamma-notions2}. We arrive at a central observation.

\begin{prop}\label{minmax-thm} The isometry $\Phi:  L^2\Omega^*(M_\Gamma)\to L^2\Omega^*(M, \mathscr{E}) $ of 
$\mathscr{N}\Gamma$-Hilbert modules in \eqref{phi-isometry}, defines isomorphisms of 
the $\mathscr{N}\Gamma$-Hilbert complexes
\begin{equation}
\begin{split}
\Phi:  (\dom_{\min}(M_\Gamma), d_\Gamma) \to (\dom_{\min}(M, \mathscr{E}), d_{\mathscr{E}}) \\
\Phi:  (\dom_{\max}(M_\Gamma), d_\Gamma)\to (\dom_{\max}(M, \mathscr{E}), d_{\mathscr{E}}).
\end{split}
\end{equation}
\end{prop}

\begin{proof}
One can check easily by construction that on $\Omega^*_c(M,\mathscr{E}_c)$ 
(in fact also on $\Omega^*_{\max}(M,\mathscr{E})$ by a partition of unity argument)
$$
d_{\mathscr{E}} = \Phi \circ d_{\Gamma} \circ \Phi^{-1}.
$$
Thus, using the isomorphisms \eqref{phi-isomorphism} and the fact that
the map $\Phi: L^2\Omega^*(M, \mathscr{E}) \to L^2\Omega^*(M_\Gamma)$ in \eqref{phi-isometry} is an isometry, we obtain
\begin{equation}
\begin{split}
\dom_{\min}(M, \mathscr{E}) = \Phi (\dom_{\min}(M_\Gamma)), \quad 
d_{\mathscr{E}} = \Phi \circ d_{\Gamma} \circ \Phi^{-1} \ \textup{on} \ \dom_{\min}(M, \mathscr{E}), \\
\dom_{\max}(M, \mathscr{E}) = \Phi (\dom_{\max}(M_\Gamma)),  \quad 
d_{\mathscr{E}} = \Phi\circ d_{\Gamma} \circ \Phi^{-1} \ \textup{on} \ \dom_{\max}(M, \mathscr{E}).
\end{split}
\end{equation}
The statement now follows. \end{proof}

\section{Existence of $L^2$-Betti numbers and Novikov-Shubin invariants}\label{finiteness-section}
\subsection{Some analytic properties}
Let $D_{\textup{rel}}$ and $D_{\textup{abs}}$ denote the self-adjoint operators 
on $L^2\Omega^*(M,g)$ defined as the rolled-up operators of the Hilbert complexes  
$(\dom_{\min}(M), d)$ and $(\dom_{\max}(M), d)$, respectively.
Let us also introduce the operators $\Delta_{\mathrm{abs}}:=D^2_{\mathrm{abs}}$ and 
$\Delta_{\mathrm{rel}}:=D^2_{\mathrm{rel}}$. Similarly, let $D_{\Gamma, \textup{rel}}$ and 
$D_{\Gamma, \textup{abs}}$ denote the self-adjoint operators on $L^2\Omega^*(M_\Gamma,g_{\Gamma})$  
defined as the rolled-up operators of the Hilbert complexes  $(\dom_{\min}(M_\Gamma), d_\Gamma)$ and 
$(\dom_{\max}(M_\Gamma), d_\Gamma)$, respectively. We write
\begin{align*}
&D_{\Gamma,\textup{rel}}:= d_{\Gamma, \, \min} + d^*_{\Gamma, \, \min}, \quad 
\Delta_{\Gamma,\mathrm{rel}}:=D^2_{\Gamma,\mathrm{rel}}, \\
&D_{\Gamma,\textup{abs}}:= d_{\Gamma, \, \max} + d^*_{\Gamma, \, \max}, \quad 
\Delta_{\Gamma,\mathrm{abs}}:=D^2_{\Gamma,\mathrm{abs}}.
\end{align*}
We define the following sets of functions 
\begin{align*}
&\mathcal{A}(\overline{M}):=\{f\in C(\overline{M}) :  f|_{M}\in C^{\infty}(M),\ df\in L^{\infty}\Omega^1(M,g)\}, \\
&\mathcal{A}(\overline{M}_{\Gamma}):=\{f\in C(\overline{M}_{\Gamma}) : 
f|_{M_{\Gamma}}\in C^{\infty}(M_{\Gamma}),\ df\in L^{\infty}_{\mathrm{loc}}\Omega^1(M_{\Gamma},g_{\Gamma})\}.
\end{align*}
Note that given an open cover $\{U_\alpha\}_{\alpha \in I}$ of $\overline{M}_{\Gamma}$, 
there exists a partition of unity $\{\phi_{\alpha'}\}_{\alpha'\in J}$, with compact support subordinate to 
$\{U_\alpha\}_{\alpha \in I}$ with $\phi_{\alpha'}\in \mathcal{A}(\overline{M}_{\Gamma})$ for each 
$\alpha'\in J$, see \cite[Proposition 3.2.2]{You}. Obviously the analogous result holds true on 
$\overline{M}$ with respect to $\mathcal{A}(\overline{M})$.

\begin{lemma}
\label{lemma-1}
Let $U\subset \overline{M}_{\Gamma}$ be an open subset such that $p|_U:U\rightarrow V$ 
is an isomorphism, with $V=p(U)$ and $p:\overline{M}_{\Gamma}\rightarrow \overline{M}$  
the covering map. Let $n \in \mathbb{N}$ and consider any collection 
$\{\phi_1, \ldots, \phi_n\} \subset \mathcal{A}(\overline{M}_{\Gamma})$ with $\supp(\phi_k)\subset U$. Then 
\begin{equation}
\label{composition}
\prod_{k=1}^n\left(\phi_k(i+D_{\Gamma,\mathrm{abs}/\mathrm{rel}})^{-1}\right)
\end{equation}
 is a Hilbert-Schmidt operator for $n\gg 0$ sufficiently large.
\end{lemma}

\begin{proof} 
Let us denote with $\Psi:L^2\Omega^*(\mathrm{reg}(U),g_{\Gamma}|_{\mathrm{reg}(U)})\rightarrow 
L^2\Omega^*(\mathrm{reg}(V),g|_{\mathrm{reg}(V)})$ the isometry induced by $(p|_U)^{-1}:V\rightarrow U$. 
Here, we write $\mathrm{reg}(U)$ and $\mathrm{reg}(V)$ for the regular part of $U$ and $V$, respectively.
Let $\phi\in \mathcal{A}(\overline{M}_{\Gamma})$ with $\supp(\phi)\subset U$ and 
$\omega\in \mathscr{D}(D_{\Gamma,\mathrm{abs}/\mathrm{rel}})$ be arbitrarily fixed. It is not difficult to show that 
\begin{align*}
&\phi \omega\in \mathscr{D}(D_{\Gamma,\mathrm{abs}/\mathrm{rel}}), \quad \Psi(\phi\omega) \in \mathscr{D}(D_{\mathrm{abs}/\mathrm{rel}}), 
\\ &D_{\mathrm{abs}/\mathrm{rel}}(\Psi(\phi\omega))=\Psi(D_{\Gamma,\mathrm{abs}/\mathrm{rel}}(\phi \omega)).
\end{align*} 
Hence we get
$$
\begin{aligned}
\phi(i+D_{\Gamma,\mathrm{abs}/\mathrm{rel}})^{-1}&=
\Psi^{-1}(i+D_{\mathrm{abs}/\mathrm{rel}})^{-1} \circ (i+D_{\mathrm{abs}/\mathrm{rel}})\Psi \circ \phi(i+D_{\Gamma,\mathrm{abs}/\mathrm{rel}})^{-1}\\
&=\Psi^{-1}(i+D_{\mathrm{abs}/\mathrm{rel}})^{-1} \circ \Psi(i+D_{\Gamma,\mathrm{abs}/\mathrm{rel}}) \circ \phi(i+D_{\Gamma,\mathrm{abs}/\mathrm{rel}})^{-1}.
\end{aligned}
$$

Note that $(i+D_{\Gamma,\mathrm{abs}/\mathrm{rel}})\phi(i+D_{\Gamma,\mathrm{abs}/\mathrm{rel}})^{-1}$ is a bounded operator and $(i+D_{\mathrm{abs}/\mathrm{rel}})^{-1}$ lies in the $n$-Schatten class for $n$ sufficiently big as a consequence of \cite[Theorem 1.1]{Alvarez}.  Therefore 
$\phi(i+D_{\Gamma,\mathrm{abs}/\mathrm{rel}})^{-1}$ is $n$-Schatten and so we can conclude that \eqref{composition} is 
Hilbert Schmidt for $n$ sufficiently big.
\end{proof}

\noindent Below, we shall impose an analytic assumption.
\begin{assump}\label{fundamental}
For any arbitrarily fixed cutoff functions $\phi$, $\chi\in \mathcal{A}(\overline{M})$ 
such that $\supp(\phi)\cap \supp(\chi)=\varnothing$, the following composition is Hilbert-Schmidt
\begin{equation}
\phi(i+D_{\mathrm{abs}/\mathrm{rel}})^{-1}\chi.
\end{equation}
\end{assump}

We expect the Assumption \ref{fundamental} to hold on any compact
smoothly stratified (Thom-Mather) pseudo-manifold with an iterated wedge metric.
Currently, we can show that the assumption holds in the following two cases. 

\begin{prop}\label{witt-examples}
The above Assumption \ref{fundamental} is satisfied in the following two cases:
\begin{enumerate}
\item $(M,g)$ has an isolated conical singularity.
\item $M$ is a Witt pseudo-manifold with (non-isolated) singularities 
of arbitrary depth and $g$ is an adapted iterated  wedge metric as in 
\cite[Proposition 5.4]{package} (notice that it is always possible to endow 
a Witt pseudo-manifold with such a metric, see again \cite[Proposition 5.4]{package}).
\end{enumerate}
\end{prop}

\begin{proof} Let us abbreviate $D:= D_{\mathrm{abs}/\mathrm{rel}}$
Notice that in the Witt case the operator is essentially self-adjoint, \cite[Section 5.8]{package}, 
so $D_{\mathrm{abs}} =D_{\mathrm{rel}}$ in this case.
The proof uses the microlocal analysis of the heat kernel $e^{-tD^2}$ by Mooers \cite[Theorem 4.1]{Moo} for the first case,
and the microlocal description of the resolvent $(D-\lambda)^{-1}$ under the Witt assumption by Albin and Gell-Redman 
\cite[Theorem 4.3]{AlGe} for the second case. In both cases, the microlocal analysis employs 
the blowup techniques, see for example Melrose \cite{Mel:TAP} and Mazzeo \cite{Maz:ETO}. 
Since the notion of blowups appears only within the limits of this proof, we don't attempt to 
provide a detailed presentation and instead suggest to think of blowups in terms of polar coordinates around the 
corresponding blown up submanifolds.\medskip

\noindent \underline{\emph{Proof of the first case:}} Since the Assumption \ref{fundamental} is concerned with the 
resolvent, when using Mooers \cite[Theorem 4.1]{Moo} in the first case we need to pass from the heat operator to the resolvent.
We do this using the following relation
\begin{align*}
(i-D) \int_0^\infty e^{-t(\textup{Id} + D^2)} dt = 
(i-D) (\textup{Id} + D^2)^{-1} = (i + D)^{-1}.
\end{align*}
In particular, we have a relation between the resolvent and the heat operator
\begin{align}\label{heat-resolvent}
\phi(i + D)^{-1} \chi = \int_0^\infty e^{-t} \phi (i-D) e^{-t D^2}\chi dt.
\end{align}
In what follows, we are going to explain why this is an equality of polyhomogeneous (in particular smooth in the interior) conormal Schwartz kernels. 
The asymptotics of the Schwartz kernel of the operator on the right hand side in \eqref{heat-resolvent} is conveniently described by 
the heat-space blowup of $\widetilde{M} \times \widetilde{M} \times [0,\infty)_{\sqrt{t}}$, where by a small abuse of 
notation we denote the boundary of $\widetilde{M}$ by $\partial M$, see \cite{Mel:TAP}. In the isolated case, which we consider 
here, $\partial M$ is simply the link of the cone. One first blows up 
the highest codimension corner $\partial M \times \partial M \times \{0\}$, which introduces a new boundary face 
ff. Then one blows up the temporal diagonal $\textup{diag} (\widetilde{M} \times \widetilde{M}) \times \{0\}$, which 
introduces the boundary face td. The resulting blowup is referred to as the heat space blowup $\mathscr{M}^2_h$ 
and is illustrated in Figure \ref{heat-space}. There, $x$ and $\widetilde{x}$ are boundary defining functions of the 
two copies of $\widetilde{M}$. The boundary faces rf, lf and tf correspond to the boundary faces 
$\{x=0\}, \{\widetilde{x}=0\}$ and $\{t=0\}$ before the blowup, respectively. There is also a canonical blowdown map 
$$
\beta: \mathscr{M}^2_h \to \widetilde{M} \times \widetilde{M} \times [0,\infty).
$$

\begin{figure}[h]
\includegraphics[scale=0.6]{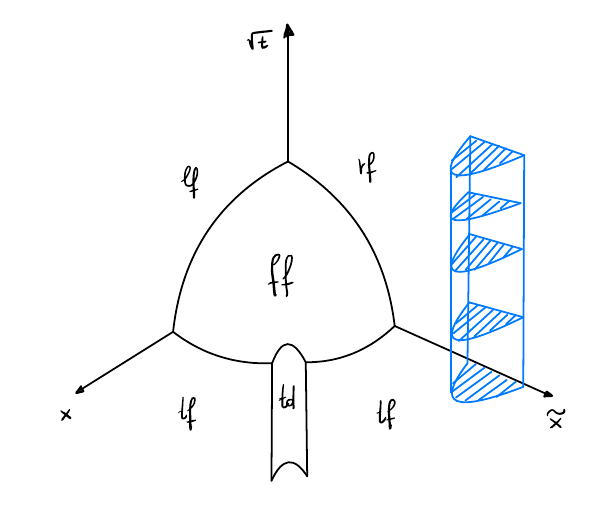}
\caption{The heat-space blowup $\mathscr{M}^2_h$ and support of $\phi e^{-t D^2} \chi$.}
\label{heat-space}
\end{figure}

Since $\supp(\phi)\cap \supp(\chi)=\varnothing$, the lift $\beta^*\phi e^{-t D^2} \chi$ is in fact supported away
from ff and td. Figure \ref{heat-space} illustrates the support of $\beta^*\phi e^{-t D^2} \chi$,
where we consider the generic case that $\supp(\phi) \cap \partial M \neq \varnothing$. In that case $\supp(\chi)$
must be fully contained in the interior of $M$ and hence the support of $\beta^*\phi e^{-t D^2} \chi$ in $\mathscr{M}^2_h$
is of the form as indicated by the blue shaded region. \smallskip

\noindent
It is a central result of Mooers \cite[Theorem 4.1]{Moo} that $\beta^*e^{-t D^2}$ is a conormal polyhomogeneous
section of $\beta^*E$ on $\mathscr{M}^2_h$, smooth in the open interior, where we have abbreviated
$$
E:=\Lambda^* ({}^{w}T^*M) \boxtimes \Lambda^* ({}^{w}T^*M).
$$ 
The kernel $\beta^* \phi e^{-t D^2} \chi$
vanishes identically near ff and td and hence the precise
asymptotics of $\beta^*e^{-t D^2}$ at ff and td is irrelevant here. The lift $\beta^*e^{-t D^2}$ is vanishing to infinite 
order at tf and is of order $(\alpha,p)$ in its asymptotics at rf, i.e. 
\begin{equation}\label{ap}
\beta^*e^{-t D^2} \sim \rho_{\textup{rf}}^\alpha \log^p(\rho_{\textup{rf}}), \ \rho_{\textup{rf}} \to 0.
\end{equation}
The precise value of $(\alpha,p)$ is governed by the spectrum of some operators on the link. 
Instead of making it explicit here, note that \eqref{ap} translates for any 
$u \in C^\infty_c(M, \Lambda^* ({}^{w}T^*M))$ and any fixed $t>0$ into an asymptotics
\begin{equation}\label{ap2}
(e^{-t D^2} u)(x) = x^\alpha \log^p(x) \, G(x), 
\end{equation}
where $G(x)$ is a bounded section of $\Lambda^* ({}^{w}T^*M)$ (with the fibrewise inner product defined by $g$)
as $x \to 0$. By definition, $e^{-t D^2} u \in \dom (D^2)$ and hence in particular 
$e^{-t D^2} u \in L^2(M, \Lambda^* ({}^{w}T^*M), g)$. Noting that
the volume form of $g$ is of the form $x^{\dim \partial M} dx$ times a volume form on $\partial M$, up to some bounded
function on $\widetilde{M}$,
$e^{-t D^2} u \in L^2$ and \eqref{ap2} implies
\begin{equation}\label{ap3}
2\alpha > - \dim \partial M -1.
\end{equation}
In the generic case, illustrated in Figure \ref{heat-space}, $\beta^* \phi e^{-t D^2} \chi$ 
is vanishing identically near ff and td, and hence it is a lift of a polyhomogeneous function on 
$[0,\infty) \times \widetilde{M} \times \widetilde{M}$. We conclude from \eqref{ap3}
\begin{align*}
\phi e^{-t D^2} \chi &\in L^2(M \times M, E; g), \\
\phi (i-D) e^{-t D^2}\chi &\in L^2(M \times M, E; g),
\end{align*}
uniformly as $t \to \infty$ and $t\to 0$. Note that for the second statement
we used exactly the same argument as above: polyhomogeneity of the lift $\beta^* (i-D) e^{-t D^2}$, and 
$(i-D) e^{-t D^2}u \in \dom(D) \subset L^2(M, \Lambda^* ({}^{w}T^*M),g)$
to get a bound of the form \eqref{ap3} for its asymptotics at rf.
Consequently the Schwartz kernel $e^{-t} \phi (i-D) e^{-t D^2}\chi $ is integrable in $t \in (0,\infty)$ and hence
$$
\int_0^\infty e^{-t} \phi (i-D) e^{-t D^2}\chi dt \in L^2(M \times M, E; g),
$$ 
We remark, that it is essential for the cutoff functions $\phi$ and $\chi$ to have disjoint support, 
since otherwise the integrand behaves as $t^{-\dim M}$ on the diagonal and is not integrable.
Thus, \eqref{heat-resolvent} holds as equality of integral kernels and 
$\phi(i + D)^{-1} \chi$ is Hilbert-Schmidt, proving the claim in the
first case.\medskip

\noindent \underline{\emph{Proof of the second case:}} 
For the second case let us first consider the case where $\overline{M}$ is stratified of depth one.
Then its resolution $\widetilde{M}$ is a compact manifold with fibered boundary, denoted by $\partial M$ with a small abuse of notation.
$\partial M$ is the total space of a fibration $\phi: \partial M \to B$ over a compact
base manifold $B$ with fibres $F$ being compact manifolds as well. The wedge metric is conical on the fibres of
the boundary collar $\mathscr{U} \cong (0,1) \times \partial M$. In this setting the resolvent $(i + D)^{-1}$ is conveniently 
described as a polyhomogeneous (with bounds)
distribution on the edge space blowup of $\widetilde{M} \times \widetilde{M}$, obtained by 
blowing up the fibre diagonal 
$$
\textup{diag}_\phi (\partial M \times \partial M) = \{(p,p') | \phi(p) = \phi(p')\}.
$$
The blowup introduces a new boundary face ff and is illustrated in Figure \ref{edge-space}.
There, $x$ and $\widetilde{x}$ are boundary defining functions of the 
two copies of $\widetilde{M}$. Moreover, $(y,\widetilde{y})$ are local coordinates on the two copies of 
the base of the fibration $\phi: \partial M \to B$, so that locally $\textup{diag}_\phi (\partial M \times \partial M) = \{y=\widetilde{y}\}$.
The boundary faces rf and lf correspond to the boundary faces 
$\{x=0\}$ and $\{\widetilde{x}=0\}$ before the blowup, respectively. There is also a canonical blowdown map 
$$
\beta: \mathscr{M}^2_e \to \widetilde{M} \times \widetilde{M}.
$$

\begin{figure}[h]
\includegraphics[scale=0.6]{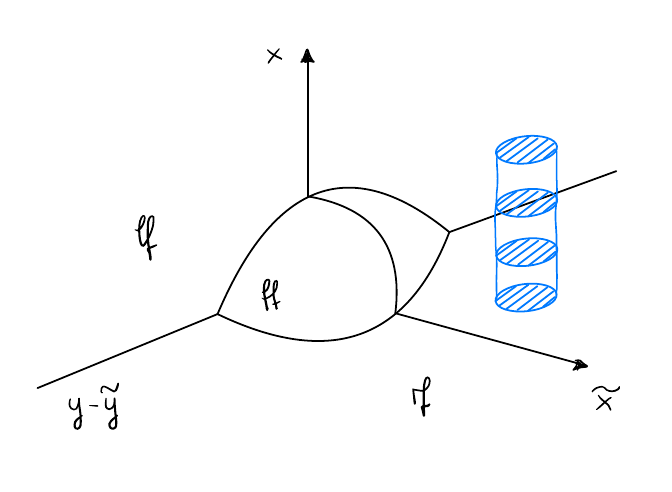}
\caption{The edge-space blowup $\mathscr{M}^2_e$ and support of $\phi e^{-t D^2} \chi$.}
\label{edge-space}
\end{figure}

\noindent
Since $\supp(\phi)\cap \supp(\chi)=\varnothing$, the lift $\beta^*\phi (i + D)^{-1} \chi$ is supported away 
from ff. Figure \ref{edge-space} also illustrates the support of $\beta^*\phi (i + D)^{-1} \chi$,
where we consider the generic case that $\supp(\phi) \cap \partial M \neq \varnothing$. In that case, 
support of $\beta^*\phi (i + D)^{-1} \chi$ in $\mathscr{M}^2_e$
is of the form as indicated by the blue shaded region. The shaded region may intersect the boundary face 
lf, but most importantly is supported away from the front face ff.  \medskip

Remark now that since 
we are considering an {\it adapted} iterated wedge metric, the operator $D$ satisfies what Albin and Gell-Redman
call the geometric Witt condition; this means that  we can indeed apply their results.
The microlocal description of the resolvent in Albin and Gell-Redman 
\cite[Theorem 4.3]{AlGe} asserts that the resolvent $(i + D)^{-1}$ lifts to a polyhomogeneous (with bounds)
distribution on $\mathscr{M}^2_e$ with a conormal singularity along the lifted diagonal.
In the generic case illustrated in Figure \ref{edge-space}, the lift
$\beta^* \phi (i + D)^{-1} \chi$ vanishes identically near ff and the lifted diagonal. 
Hence $\beta^* \phi (i + D)^{-1} \chi$ is smooth in the interior of 
$\mathscr{M}^2_e$, vanishing to infinite order at ff. \medskip

Let us now discuss the asymptotics of $\beta^* (i + D)^{-1}$ and $\beta^* \phi (i + D)^{-1} \chi$ 
at the left face lf and the right face rf. By symmetry, it suffices to study the right face rf only.
The asymptotics of $\beta^*(i + D)^{-1}$ (and hence also of $\beta^* \phi (i + D)^{-1} \chi$) at rf has a lower bound $\alpha'$, i.e. 
\begin{equation}\label{ap4}
\rho_{\textup{rf}}^{-\alpha'} \cdot  \beta^*(i + D)^{-1} 
\end{equation}
is a section of $\beta^*E$ (with the fibrewise inner product defined by $g$),
bounded near $\rho_{\textup{rf}} = 0$. The bound $\alpha'$
is determined by the spectrum of certain operators on the fibres $F$.
Instead of inferring the explicit value of $\alpha'$ from \cite[Theorem 4.3]{AlGe}, we argue as in the isolated case. Let us note that 
\eqref{ap4} translates for any $u \in C^\infty_c(M,\Lambda^* ({}^{w}T^*M))$ into
\begin{equation}\label{ap5}
x^{-\alpha'} \Bigl((i + D)^{-1} u\Bigr)(x)
\end{equation}
being bounded near $x = 0$ as a section of $\Lambda^* ({}^{w}T^*M)$ with the fibrewise inner product defined by $g$.
By definition, $(i + D)^{-1} u \in \dom (D)$ and hence in particular $(i + D)^{-1} u \in L^2(M, \Lambda^* ({}^{w}T^*M),g)$.
Exactly as above in \eqref{ap3}, this implies 
\begin{equation}\label{ap6}
2\alpha' > - \dim F -1.
\end{equation}
Since $\beta^* \phi (i + D)^{-1} \chi$ vanishes identically near ff and is smooth in the interior, 
it is the lift of a polyhomogeneous function on $\widetilde{M}\times \widetilde{M}$
with bound $\alpha'$
in its asymptotics as $x, \widetilde{x} \to 0$. Hence by \eqref{ap6}
\begin{align*}
\phi (i + D)^{-1} \chi \in L^2(M \times M, E; g).
\end{align*}
This proves the claim in the second case in case of stratification depth one.
In the general stratification depth, the front face ff in $\mathscr{M}^2_e$ in Figure \ref{edge-space} 
requires additional blowups, see \cite[Definition 3.1]{AlGe}. However, $\beta^*\phi (i + D)^{-1} \chi$
is supported away from ff and thus these blowups do not affect the argument. This completes the proof.
 \end{proof}

Before we continue to the next result, we shall point out that
in the Witt case and for an adapted metric as in the second case of Proposition \ref{witt-examples}, 
the Gauss-Bonnet operator is essentially self-adjoint on $(M_\Gamma,g_\Gamma)$,
see \cite[Proposition 6.3]{package}. Hence in that case there is no point in 
distinguishing $D_{\Gamma, \mathrm{abs}}$ and $D_{\Gamma, \mathrm{rel}}$.

\begin{prop}
\label{key-trick}
Let $U\subset \overline{M}_{\Gamma}$ be an open subset, 
such that $p|_U:U\rightarrow V$ is an isomorphism, with $V=p(U)$ and $p:\overline{M}_{\Gamma}\rightarrow \overline{M}$ 
the covering map. Let $\overline{\phi}$, $\overline{\chi}\in \mathcal{A}(\overline{M}_{\Gamma})$ with  $\supp(\overline{\phi})\cap \supp(\overline{\chi})=\varnothing$ and $\supp(\overline{\phi})\subset U$. If Assumption \ref{fundamental} holds true then $$\overline{\phi}(i+D_{\Gamma, \mathrm{abs}/\mathrm{rel}})^{-1}\overline{\chi}$$ is a Hilbert-Schmidt operator.
\end{prop}

\begin{proof}
Let $\Psi$ be the isometry defined in the proof of Lemma \ref{lemma-1}.
Let $\overline{\theta}\in \mathcal{A}(\overline{M}_{\Gamma})$ with $\supp(\overline{\theta})\subset U$, 
$\supp(\overline{\theta})\cap \supp(\overline{\chi})=\varnothing$ and $\overline{\theta} \equiv 1$ on an open neighbourhood of $\supp(\overline{\phi})$. 
Since $\overline{\theta}\in \mathcal{A}(\overline{M}_{\Gamma})$ the following operator $[D_{\Gamma,\mathrm{abs}/\mathrm{rel}},\overline{\theta}](i+D_{\Gamma, \mathrm{abs}/\mathrm{rel}})^{-1}\overline{\chi}$ is well defined and we have
\begin{equation}\label{start}
\begin{split}
&[D_{\Gamma,\mathrm{abs}/\mathrm{rel}},\overline{\theta}](i+D_{\Gamma, \mathrm{abs}/\mathrm{rel}})^{-1}\overline{\chi}\\
= \, &D_{\Gamma,\mathrm{abs}/\mathrm{rel}}\overline{\theta}(i+D_{\Gamma, \mathrm{abs}/\mathrm{rel}})^{-1}\overline{\chi}-\overline{\theta}(-i+i+D_{\Gamma, \mathrm{abs}/\mathrm{rel}})(i+D_{\Gamma, \mathrm{abs}/\mathrm{rel}})^{-1}\overline{\chi}\\
= \, &(i+D_{\Gamma, \mathrm{abs}/\mathrm{rel}})\overline{\theta}(i+D_{\Gamma, \mathrm{abs}/\mathrm{rel}})^{-1}\overline{\chi}-\overline{\theta}\overline{\chi}\\
= \, &(i+D_{\Gamma, \mathrm{abs}/\mathrm{rel}})\overline{\theta}(i+D_{\Gamma, \mathrm{abs}/\mathrm{rel}})^{-1}\overline{\chi}.
\end{split}
\end{equation}
Let now introduce a second auxiliary function $\overline{\zeta}\in \mathcal{A}(\overline{M}_{\Gamma})$ 
with $\supp(\overline{\zeta})\subset U$ and $\overline{\zeta}\equiv 1$ on $\supp([D_{\Gamma,\mathrm{abs}/\mathrm{rel}},\overline{\theta}])$. 
Then $[D_{\Gamma,\mathrm{abs}/\mathrm{rel}},\overline{\theta}]\overline{\zeta} \equiv [D_{\Gamma,\mathrm{abs}/\mathrm{rel}},\overline{\theta}]$.
Applying $\Psi$ on both sides of \eqref{start} implies
$$
\begin{aligned}
\Psi[D_{\Gamma,\mathrm{abs}/\mathrm{rel}},\overline{\theta}]\overline{\zeta}(i+D_{\Gamma, \mathrm{abs}/\mathrm{rel}})^{-1}\overline{\chi}&=\Psi[D_{\Gamma,\mathrm{abs}/\mathrm{rel}},\overline{\theta}](i+D_{\Gamma, \mathrm{abs}/\mathrm{rel}})^{-1}\overline{\chi}\\&=\Psi(i+D_{\Gamma, \mathrm{abs}/\mathrm{rel}})\overline{\theta}(i+D_{\Gamma, \mathrm{abs}/\mathrm{rel}})^{-1}\overline{\chi}.
\end{aligned}
$$ 
Let us denote  $\theta:=\overline{\theta} \circ (p|_U)^{-1}$.
The presence of $\overline{\zeta}$ allows us to commute $\Psi$ with the commutator and yields
$$
[D_{\mathrm{abs}/\mathrm{rel}},\theta]\Psi\overline{\zeta}(i+D_{\Gamma, \mathrm{abs}/\mathrm{rel}})^{-1}\overline{\chi}=(i+D_{\mathrm{abs}/\mathrm{rel}})\Psi\overline{\theta}(i+D_{\Gamma, \mathrm{abs}/\mathrm{rel}})^{-1}\overline{\chi}.
$$ 
Applying the resolvent on both sides of this equality implies
$$
(i+D_{\mathrm{abs}/\mathrm{rel}})^{-1}[D_{\mathrm{abs}/\mathrm{rel}},\theta]\Psi\overline{\zeta}(i+D_{\Gamma, \mathrm{abs}/\mathrm{rel}})^{-1}\overline{\chi}=\Psi\overline{\theta}(i+D_{\Gamma, \mathrm{abs}/\mathrm{rel}})^{-1}\overline{\chi}.
$$ 
Therefore, by multiplying both sides with $\phi:=\overline{\phi} \circ (p|_U)^{-1}$, we get 
$$
\begin{aligned}
\phi(i+D_{\mathrm{abs}/\mathrm{rel}})^{-1}[D_{\mathrm{abs}/\mathrm{rel}},\theta]\Psi\overline{\zeta}(i+D_{\Gamma, \mathrm{abs}/\mathrm{rel}})^{-1}\overline{\chi}=\phi\Psi\overline{\theta}(i+D_{\Gamma, \mathrm{abs}/\mathrm{rel}})^{-1}\overline{\chi}\\
=\Psi\overline{\phi}\, \overline{\theta}(i+D_{\Gamma, \mathrm{abs}/\mathrm{rel}})^{-1}\overline{\chi}
=\Psi\overline{\phi}(i+D_{\Gamma, \mathrm{abs}/\mathrm{rel}})^{-1}\overline{\chi}
\end{aligned}
$$  
and so we arrive at
$$
\Psi^{-1}\phi(i+D_{\mathrm{abs}/\mathrm{rel}})^{-1}
[D_{\mathrm{abs}/\mathrm{rel}},\theta]\Psi\overline{\zeta}
(i+D_{\Gamma, \mathrm{abs}/\mathrm{rel}})^{-1}\overline{\chi}=
\overline{\phi}(i+D_{\Gamma, \mathrm{abs}/\mathrm{rel}})^{-1}\overline{\chi}.
$$ 
Note now that $\supp(\phi)\subset V$, $\phi\in \mathcal{A}(\overline{M})$ and 
$\supp(\phi)\cap \supp([D_{\mathrm{abs}/\mathrm{rel}},\theta])=\varnothing$. 
Hence there exists $\gamma\in \mathcal{A}(\overline{M})$ such that 
$\gamma [D_{\mathrm{abs}/\mathrm{rel}},\theta]=[D_{\mathrm{abs}/\mathrm{rel}},\theta]$ and 
$\supp(\phi)\cap \supp(\gamma)=\varnothing$. We obtain
$$
\Psi^{-1}\phi(i+D_{\mathrm{abs}/\mathrm{rel}})^{-1}\gamma
[D_{\mathrm{abs}/\mathrm{rel}},\theta]\Psi\overline{\zeta}(i+D_{\Gamma, \mathrm{abs}/\mathrm{rel}})^{-1}
\overline{\chi}=\overline{\phi}(i+D_{\Gamma, \mathrm{abs}/\mathrm{rel}})^{-1}\overline{\chi}
$$
Finally since both $[D_{\mathrm{abs}/\mathrm{rel}},\theta]$ and 
$(i+D_{\Gamma, \mathrm{abs}/\mathrm{rel}})^{-1}$ are bounded operators 
and $\phi(i+D_{\mathrm{abs}/\mathrm{rel}})^{-1}\gamma$ is a Hilbert-Schmidt operator by Assumption \ref{fundamental}, 
we can conclude that $\overline{\phi}(i+D_{\Gamma, \mathrm{abs}/\mathrm{rel}})^{-1}\overline{\chi}$ is also Hilbert-Schmidt, as required.
\end{proof}

\begin{prop}
In the setting of Proposition \ref{key-trick} the operator 
$$
(\textup{Id}+\Delta_{\Gamma,\mathrm{abs}/\mathrm{rel}})^{-n}
$$ 
is $\Gamma$-trace class for $n \gg 0$ sufficiently large.
\end{prop}

\begin{proof}
Since the argument is the same for both $\Delta_{\Gamma,\mathrm{abs}}$ 
and $\Delta_{\mathrm{rel}}$ we drop the subscription $\mathrm{abs}/\mathrm{rel}$ in 
the rest of the proof. Now we note that $(\textup{Id}+\Delta_{\Gamma})^{-n}=
(i+D_{\Gamma})^{-n/2}\circ(i-D_{\Gamma})^{-n/2}$. Thus it is enough to prove that both 
$(i+D_{\Gamma})^{-n/2}$ and $(i-D_{\Gamma})^{-n/2}$ are $\Gamma$-Hilbert-Schmidt. 
We show now that $(i+D_{\Gamma})^{-n/2}$ is  $\Gamma$-Hilbert-Schmidt. The proof of the 
other case is identical. According to \cite[Definition (4.3)']{Ati} we need to show that $\theta(i+D_{\Gamma})^{-n/2}$ 
is Hilbert-Schmidt, with $\theta$ an arbitrarily fixed bounded measurable function with compact support on 
$\overline{M}_{\Gamma}$. Let $\{U_1,...,U_{\ell}\}$ be a finite open cover of $\supp(\theta)$ such that $p|_{U_k}:U_k\rightarrow V_k=p(U_k)$ is an isomorphism. Thanks to \cite[Proposition 3.2.2]{You} we can find $\psi_1,...,\psi_k\in \mathcal{A}(\overline{M}_{\Gamma})$ such that  $\supp(\psi_k)\subset U_k$ and $\sum_{k=1}^{\ell}\psi_k(x)=1$ for each $x\in \supp(\theta)$. Thus we have
$$
\begin{aligned}
\theta(i+D_{\Gamma})^{-n/2}=\theta\left(\sum_{k=1}^{\ell}\psi_k\right)(i+D_{\Gamma})^{-n/2}=\sum_{k=1}^{\ell}\theta\psi_k(i+D_{\Gamma})^{-n/2}.
\end{aligned}
$$
Since $\theta$ is bounded, it is enough to show that $\psi_k(i+D_{\Gamma})^{-n/2}$ is $\Gamma$-Hilbert-Schmidt for any $k$. 
Without loss of generality we can assume that $n$ is even. Let $q=n/2$ and  $\phi_1\in \mathcal{A}(\overline{M}_{\Gamma})$ 
with $\supp(\phi_1)\subset U_k$ and $\supp(\psi_k)\cap \supp(1-\phi_1)=\varnothing$. We have
$$
\begin{aligned}
\psi_k(i+D_{\Gamma})^{-n/2}&=\psi_k(i+D_{\Gamma})^{-1}(i+D_{\Gamma})^{1-q}\\
&=\psi_k(i+D_{\Gamma})^{-1} (\phi_1+1-\phi_1)(i+D_{\Gamma})^{1-q}\\
&= \psi_k(i+D_{\Gamma})^{-1}\phi_1(i+D_{\Gamma})^{1-q}+\psi_k(i+D_{\Gamma})^{-1}(1-\phi_1)(i+D_{\Gamma})^{1-q}. 
\end{aligned}
$$
Note that $\psi_k(i+D_{\Gamma})^{-1}(1-\phi_1)(i+D_{\Gamma})^{1-q}$ is 
Hilbert-Schmidt, since $(i+D_{\Gamma})^{1-q}$ is bounded and $\psi_k(i+D_{\Gamma})^{-1}(1-\phi_1)$ is 
Hilbert-Schmidt thanks to Proposition \ref{key-trick}. Thus $\psi_k(i+D_{\Gamma})^{-p/2}$ is Hilbert-Schmidt if and only if $\psi_k(i+D_{\Gamma})^{-1}\phi_1(i+D_{\Gamma})^{1-q}$ is Hilbert-Schmidt. By picking
$\phi_1,...,\phi_{q-1}\in \mathcal{A}(\overline{M}_{\Gamma})$ with $\supp(\phi_j)\subset U_k$ and 
$\supp(\phi_j)\cap \supp(1-\phi_{j+1})=\varnothing$ for each $j=2,...,q-1$ and iterating the above procedure,
we get that $\psi_k(i+D_{\Gamma})^{-n/2}$ is Hilbert-Schmidt if and only if the composition 
\begin{equation}
\label{HScomposition}
\psi_k(i+D_{\Gamma})^{-1}\phi_1(i+D_{\Gamma})^{-1}...\phi_{q-1}(i+D_{\Gamma})^{-1}
\end{equation}
is Hilbert-Schmidt. Thanks to Lemma \ref{lemma-1} we know that \eqref{HScomposition} is 
Hilbert-Schmidt. We can thus finally conclude that $\psi_k(i+D_{\Gamma})^{-n/2}$ is also Hilbert-Schmidt.
\end{proof}

\begin{remark}
Note that this statement in the Witt case, where the operator is in fact essentially self-adjoint,
 is already claimed in \cite[Proposition 7.4]{PiVe1}. However, the argument given there is very short
 and does not explain properly the complexity of the proof.
 \end{remark}

\begin{thm}\label{trace-class-bottom-up} Let $f:\R \to \R$ be any rapidly decreasing function. Then
 $$f(\Delta_{\Gamma, \mathrm{abs}/\mathrm{rel}})$$ is $\Gamma$-trace class. 
In particular, the corresponding heat operators
as well as spectral projections to finite intervals are $\Gamma$-trace class.
\end{thm}

\begin{proof}
Let us consider the function $g(z):= (1+z)^n f(z)$ with $n>0$ sufficiently large. Note that $g$ is a bounded function
and thus $g(\Delta_{\Gamma, \mathrm{abs}/\mathrm{rel}})$ is a bounded operator. Writing $f(z)=(1+z)^{-n} g(z)$ we conclude that 
$$
\begin{aligned}
f(\Delta_{\Gamma, \mathrm{abs}/\mathrm{rel}})&= (\textup{Id}+\Delta_{\Gamma, \mathrm{abs}/\mathrm{rel}})^{-n} \circ 
\Bigl((\textup{Id}+\Delta_{\Gamma, \mathrm{abs}/\mathrm{rel}})^{n}f(\Delta_{\Gamma, \mathrm{abs}/\mathrm{rel}})\Bigr)\\
&=(1+\Delta_{\Gamma, \mathrm{abs}/\mathrm{rel}})^{-n} g(\Delta_{\Gamma, \mathrm{abs}/\mathrm{rel}}),
\end{aligned}
$$ 
is the composition of a bounded operator $g(\Delta_{\Gamma, \mathrm{abs}/\mathrm{rel}})$ with a 
$\Gamma$-trace class operator $(\textup{Id}+\Delta_{\Gamma, \mathrm{abs}/\mathrm{rel}})^{-n}$. 
Thus $f(\Delta_{\Gamma, \mathrm{abs}/\mathrm{rel}})$ is $\Gamma$-trace class, as required.
\end{proof}

We have now everything in place to define $L^2$-Betti numbers and 
Novikov-Shubin invariants on $\overline{M}_\Gamma$, thereby proving our first main Theorem \ref{main1}. \medskip

\subsection{Definition of $L^2$-Betti numbers and Novikov-Shubin invariants}

We begin by recalling the notion of 
$\Gamma$-dimension of a   Hilbert $\mathscr{N}\Gamma$-module.  As before, we refer the reader
to \cite[\S 1 and \S 2]{Lueck-book} and \cite{Schick} for more details. 
Continuing in the notion of Definition \ref{gamma-notions1}, we have the following.

\begin{defn}\label{gamma-notions3} Recall, a $\mathscr{N}\Gamma$-Hilbert module $H$ is a Hilbert space 
with an isometric linear $\Gamma$-embedding of $H$ into the tensor product $\mathcal{H}\otimes \ell^2 \Gamma$. 
The $\Gamma$ dimension of a Hilbert $\mathscr{N}\Gamma$-module $H$
is defined as the von Neuman trace of the orthogonal projection 
$\mathcal{P}:\mathcal{H}\otimes L^2(\Gamma)\rightarrow H$. Explicitly, let $\{b_i\}$ be an 
arbitrarily fixed Hilbert basis of $\mathcal{H}$ and $e\in \Gamma$ the unit element, 
see \cite[Definition 1.8]{Lueck-book}. Then we have 
\begin{align}\label{gamma-dim}
\dim_{\Gamma}(H):= \mathrm{Tr}_{\Gamma}(\mathcal{P})=
\sum_i\langle \mathcal{P}(b_i\otimes e),b_i\otimes \delta_e\rangle_{\mathcal{H}\otimes L^2(\Gamma)} \in [0,\infty].
\end{align}
\end{defn}
As explained in \cite[p. 17]{Lueck-book} the above definition is well-posed and independent of the choices made.
Now let us consider such an $\mathscr{N}\Gamma$-complex $H_\bullet:=(\dom(d),d)$
and define in terms of the adjoint $d^*$ the associated Laplace operator 
\begin{align}
\Delta_k:=d_k^*\circ d_k+d_{k-1}\circ d_{k-1}^*.
\end{align}
It is clear that $\Delta_k$ commutes with the $\Gamma$-action and thus 
$\ker(\Delta_k)$ has the structure of  Hilbert $\mathscr{N}\Gamma$-module.
We can therefore define in view of \eqref{gamma-dim}.

\begin{defn}
The k-th $L^2$-Betti number of $H_\bullet:=(\dom(d),d)$ is defined as 
\begin{align}
b^k_{(2),\Gamma}(H_{\bullet}):=\dim_{\Gamma}(\ker(\Delta_k)) \in [0,\infty].
\end{align}
\end{defn}

In order to define Novikov-Shubin invariants, consider for each $k$ and $\lambda\geq 0$
\begin{equation}
\begin{split}
&d^{\perp}_k:=d_k \restriction \dom(d_k)\cap \overline{\mathrm{im}(d_{k-1})}^{\perp}, \\
&\mathcal{L}(d_k^{\perp},\lambda):= \Bigl\{ L\subset H_k \ \textup{Hilbert $\mathscr{N}\Gamma$-submodules} \ | \\
&\qquad \qquad \qquad L\subset \dom(d_k^{\perp}), \forall\ u \in L: \|d_k^{\perp}u\|\leq \lambda \|u\|\Bigr\}.
\end{split}
\end{equation}
 
\begin{defn} The Novikov-Shubin invariants are defined in two steps.

\begin{enumerate}
\item The k-th spectral density function of $H_\bullet:=(\dom(d),d)$ is defined 
as
$$
F_k(\lambda,H_{\bullet}) : [0,\infty)\rightarrow  [0, \infty] \qquad 
\lambda\mapsto \sup \Bigl\{ \dim_{\Gamma}L | L\subset \mathcal{L}(d_k^{\perp},\lambda) \Bigr\}. 
$$ 

\item We say that a $\mathscr{N}\Gamma$-Hilbert complex $H_\bullet:=(\dom(d),d)$ is Fredholm if for each 
$k$ there exists $\lambda_k$ such that $F_k(\lambda_k,H_{\bullet})<\infty$. Note that then 
$F_k(\lambda,H_{\bullet})<\infty$ for each $\lambda \in (0, \lambda_k)$ as $F_k(\lambda,H_{\bullet})$ is non-decreasing. 
In that case, we define the $k$-th Novikov-Shubin invariant as 
$$
\alpha_k(H_{\bullet}):=\lim_{\lambda\rightarrow 0^+}\frac{\log \bigl(F_{k-1}(\lambda,H_{\bullet})-F_{k-1}(0,H_{\bullet})\bigr)}{\log(\lambda)}\in [0,\infty],
$$
provided that $F_{k-1}(\lambda,H_{\bullet})-F_{k-1}(0,H_{\bullet})>0$ holds for all $\lambda>0.$ 
Otherwise, we put $\alpha_k(H_{\bullet}):=\infty^{+}$. 
\end{enumerate}
\end{defn}
In the above definition $\infty^{+}$ denotes a new formal symbol which should not be confused with $+\infty$. We have $\alpha_k(H_{\bullet})=\infty^{+}$ if and only if there is an $\epsilon>0$ such that $F_{k-1}(\epsilon,H_{\bullet})=F_{k-1}(0,H_{\bullet})$. Note that 
$$
F_k(0,H_{\bullet})=b^k_{(2),\Gamma}(H_{\bullet}).
$$
Consider now our setting of a compact smoothly (Thom-Mather) stratified pseudomanifold 
$\overline{M}$ with an iterated wedge metric $g$ on $M$. Let $\overline{M}_{\Gamma}\rightarrow \overline{M}$ be a 
Galois $\Gamma$-covering and lift $g$ to a metric $g_\Gamma$ in $M_\Gamma$. Consider the corresponding $\mathscr{N}\Gamma$-Hilbert complexes 
$(\dom_{\min}(M_\Gamma), d_\Gamma)$ and $(\dom_{\max}(M_\Gamma), d_\Gamma)$
with the corresponding Laplacians $\Delta_{\Gamma, \, \textup{rel}}$ and $\Delta_{\Gamma, \, \textup{abs}}$, respectively.
The maximal and minimal $L^2$-Betti numbers, spectral density functions and Novikov-Shubin invariants with respect to the 
Galois covering $(\overline{M}_{\Gamma},g_{\Gamma})\rightarrow (\overline{M},g)$ are defined as follows

\begin{equation}\label{invariants}
\begin{split}
b^k_{(2),\, \min/\max}(\overline{M}_{\Gamma}) &:= \dim_{\Gamma}\Bigl(\ker\bigl(\Delta_{\Gamma,\, k,\, \mathrm{rel}/\mathrm{abs}}\bigr)\Bigr), \\
F_{k,\min/\max}(\lambda,\overline{M}_{\Gamma}) &:= F_{k}\Bigl(\lambda, \bigl(\dom_{\min / \max}(M_\Gamma), d_\Gamma\bigr)\Bigr), \\
\alpha_{k,\max/\min}(\overline{M}_{\Gamma}) &:= \alpha_k\bigl(\dom_{\min / \max}(M_\Gamma), d_\Gamma\bigr).
\end{split}
\end{equation}

The definition of the Novikov-Shubin invariants $\alpha_{k,\max/\min}(\overline{M}_{\Gamma})$
makes sense only if the corresponding $\mathscr{N}\Gamma$-Hilbert complexes $(\dom_{\min / \max}(M_\Gamma), d_\Gamma)$ are Fredholm. 
This is done below, in  the proof of our first main result, Theorem \ref{main1}.

\begin{proof}[Proof of Theorem \ref{main1}] 
The first statement follows immediately from Theorem \ref{trace-class-bottom-up}. Indeed for each 
$P \in \{\Delta_{\Gamma,\, k,\, \mathrm{rel}}, \Delta_{\Gamma,\, k,\, \mathrm{abs}}\}$ let $\{E_{\lambda} (P), \lambda \in \mathbb{R}\}$ 
denote the corresponding spectral family. Theorem \ref{trace-class-bottom-up} tells us that $E_{\lambda} (P)$
is a $\Gamma$ trace-class projection for any $\lambda$. Therefore 
$$
\dim_{\Gamma}(\ker P )
= \mathrm{Tr}_{\Gamma}(E_{0}(P))<\infty.
$$ 
Thus the $L^2$-Betti numbers $b^k_{(2),\, \min/\max}(\overline{M}_{\Gamma})$ are finite.
Concerning the second part, since $E_{\lambda}(P)$ is a $\Gamma$ trace-class projection for each  
$\lambda$, the Hilbert $\mathscr{N}\Gamma$ complexes $(\dom_{\min / \max}(M_\Gamma), d_\Gamma)$ are Fredholm, 
see \cite[Proposition 2.3]{Lueck-book} and the argument given in the proof of \cite[Lemma 2.6]{Lueck-book}. We can thus 
conclude that the Novikov-Subin invariants $\alpha_{k,\max/\min}(\overline{M}_{\Gamma})$ are well-defined, as claimed.
\end{proof}

\section{A Hilsum-Skandalis-type replacement of a pullback}\label{HS-subsection}
\subsection{Topological preliminaries}
This section contains some topological properties that are  well known to people familiar with coverings. 
However, since we could not pin down a specific reference,
we prefer to collect and prove the necessary results  in this preliminary
subsection. In what follows $X$ and $Y$ stand for path-connected, locally path-connected and semi-locally simply-connected 
topological spaces. Consider a possibly disconnected Galois covering $p:N\rightarrow Y$ with $\Gamma$ the corresponding 
group of deck transformations. Let $y\in Y$ and $n\in p^{-1}(y)$. Then we denote with 
$$
\Phi_{N,y,n}:\pi_1(Y,y)\rightarrow \Gamma.
$$ 
the homomorphism of groups that assigns to each $[\alpha]\in \pi_1(Y,y)$ the unique element 
$\gamma\in \Gamma$ such that $\gamma(n)=n[\alpha]$, where 
$$
p^{-1}(y)\times \pi_1(Y,y)\rightarrow p^{-1}(y),\quad (n,[\alpha])\mapsto n[\alpha]
$$ 
denotes the monodromy action, see \cite[Ch. 13]{Manetti}. 

\begin{lem}
\label{1lemma}
Let $p:N\rightarrow Y$ be a possibly disconnected Galois covering. 
Then $N$ is connected if and only if $\Phi_{N,y,n}$ is surjective for any choice of $y\in Y$ and $n\in p^{-1}(y)$.
\end{lem}

\begin{proof}
The surjectivity of $\Phi_{N,y,n}$ provided $N$ is connected is well known, see \cite[\S 13.3]{Manetti}. 
The converse is an easy exercise that we leave to the reader.
\end{proof}

\begin{lem}
\label{lemma}
Let $f:X\rightarrow Y$ be a continuous map. Let $x\in X$ be such that $f(x)=y$
and let $f_*: \pi_1(X,x) \to \pi_1(Y,y)$ be the homomorphism induced by $f$. 
Consider the Galois $\Gamma$-covering $f^*N\rightarrow X$.
Then for any $(x,n) \in f^*N$
$$\Phi_{f^*N,x,(x,n)}=\Phi_{N,y,n}\circ f_*.$$
\end{lem}

\begin{proof}
Let $r:f^*N\rightarrow X$ the covering map of $f^*N$ and let  $\hat{f}:f^*N\rightarrow N$ 
be the usual map induced by the right projection. The maps combine into a commutative diagram

\begin{figure}[h!]
\begin{center}
\begin{tikzpicture}

\draw[->] (5,0) -- (7,0);
\node at (4.5,0) {$X$};
\node at (7.5,0) {$Y$};
\node at (6,-0.5) {$f$};

\draw[->] (5.4,2) -- (6.8,2);
\node at (4.5,2) {$f^*N$};
\node at (7.5,2) {$N$};
\node at (6,2.5) {$\hat{f}$};

\draw[->] (4.5,1.5) -- (4.5,0.5);
\draw[->] (7.5,1.5) -- (7.5,0.5);

\node at (4,1) {$r$};
\node at (7.9,1) {$p$};

\end{tikzpicture}
\end{center}
\end{figure}

\noindent First of all we note that $p\circ \hat{f}=f\circ r$ and that 
$\hat{f}:f^*N\rightarrow N$ is $\Gamma$-equivariant. Let now 
$[\alpha]\in \pi_1(X,x)$ and let $g=\Phi_{f^*N,x,(x,n)}([\alpha])$.
Then we have $g((x,n))=(x,n)[\alpha]$. Consider now $f_*([\alpha])$ and 
let $g'=\Phi_{N,y,n}(f_*([\alpha]))$. Then $g'(n)=n[f\circ \alpha]$. 
By \cite[\S 13.1]{Manetti} we know that $n[f\circ\alpha]=\hat{f}((x,n)[\alpha])$. Hence we have
$$
g'(n)=n[f\circ\alpha]=\hat{f}((x,n)[\alpha])=\hat{f}(g((x,n)))=g(\hat{f}((x,n)))=g(n).
$$ 
Therefore $g'(n)=g(n)$ and thus $g'=g$, as desired.
\end{proof}

\begin{lem}
\label{connection}
In the setting of Lemma \ref{lemma} assume that $f$ is homotopy equivalence. Then $f^*N$ is connected if and only if $N$ is so. 
\end{lem}

\begin{proof}
If $N$ is connected then $\Phi_{N,y,n}$ is surjective and consequently $\Phi_{f^*N,x,(x,n)}$ is also surjective, as $\Phi_{f^*N,x,(x,n)}=\Phi_{N,y,n}\circ f_*$. Thus, thanks to Lemma \ref{1lemma}, we can conclude that $f^*N$ is connected. Conversely let us assume that $f^*N$ is connected. Then $\Phi_{f^*N,x,(x,n)}$ is surjective and therefore $\Phi_{N,y,n}$ is also surjective as $\Phi_{f^*N,x,(x,n)}=\Phi_{N,y,n}\circ f_*$. We can thus conclude that $N$ is connected.
\end{proof}

\begin{cor}
\label{pullback-u}
Let $f:X\rightarrow Y$ be a homotopy equivalence and let $p:N\rightarrow Y$ be a universal covering of $Y$. Then $r:f^*N\rightarrow X$ is an universal covering of $X$.
\end{cor}

\begin{proof}
Thanks to Lemma \ref{connection} we know that $f^*N$ is connected. Let now  $x\in X$ and $n\in p^{-1}(f(x))$. Let $\hat{f}:f^*N\rightarrow N$ be the map induced by the right projection. We have $p\circ \hat{f}=f\circ r$ and thus $(p\circ \hat{f})_*=(f\circ r)_*$ as group homomorphism acting between $\pi_1(f^*N,(x,n))$ and $\pi_1(Y,f(x))$. Since  $(f\circ r)_*$ is injective and $\pi_1(N,n)$ is trivial we deduce that $\pi_1(f^*N,(x,n))$ is trivial, as well. We can thus conclude that $r:f^*N\rightarrow X$ is a universal covering of $X$ since it is a connected and simply connected covering of $X$.
\end{proof}

More generally we have the following result.

\begin{lem}
Let $f:X\rightarrow Y$ be a homotopy equivalence and let $q:M\rightarrow X$ and $p:N\rightarrow Y$ be two path-connected Galois $\Gamma$-covering. Assume that there exists $x\in X$, $m\in q^{-1}(x)$ and $n\in p^{-1}(f(x))$ such that 
\begin{equation}
\label{compatibility}
f_*(\ker(\Phi_{M,x,m}))=\ker(\Phi_{N,f(x),n}).
\end{equation}
 Then $M$ and $f^*N$ are isomorphic Galois $\Gamma$-coverings.
\end{lem}

\begin{proof}
First we note that both $M$ and $f^*N$ are path-connected. 
Write $r:f^*N\rightarrow X$ for the covering map of $f^*N$. Thanks to \cite[ Proposition 1.37]{Hatcher} 
we know that if $q_*(\pi_1(M,m))=r_*(\pi_1(f^*N,(x,n)))$ then $M$ and $f^*N$ are isomorphic.
Thanks to \cite[Propositiion 1.39]{Hatcher} we know that $q_*(\pi_1(M,m))=\ker(\Phi_{M,x,m})$ and $r_*(\pi_1(f^*N,(x,n)))=\ker(\Phi_{f^*N,x,(x,n)})$. 
Note that from Lemma \ref{lemma} we have 
$$
\ker(\Phi_{f^*N,x,(x,n)})=(f_*)^{-1}(\ker(\Phi_{N,f(x),n}))=
(f_*)^{-1}(f_*(\ker(\Phi_{M,x,m})))=\ker(\Phi_{M,x,m}).
$$ 
We can thus conclude that $M$ and $f^*N$ are isomorphic, as required.
\end{proof}

\begin{remark}
In \cite[p. 382]{Gromov-Shubin} it is required that $\Phi_{M,x,m}=\Phi_{N,f(x),n}\circ f_*$. 
Obviously their assumption implies \eqref{compatibility}.
\end{remark}

\subsection{Construction} 

Consider two compact smoo\-thly stratified spaces $\overline{M}$ and $\overline{N}$
with the respective iterated incomplete wedge metrics $g_M$ and $g_N$ in the open interior. 
We repeat the definition of a smoothly stratified stratum preserving map between $\overline{M}$ and $\overline{N}$.

\begin{defn}\label{stratified-map-def} \textup{(\cite[Definition 4 and \S 9.2]{package}, \cite[Definitions 2.6, 2.7]{ALMP3})}
Let  $\overline{M}$ and $ \overline{N}$
be  two smoothly stratified spaces.

\begin{enumerate}
\item A stratified map between  $\overline{M}$ and $ \overline{N}$
is a continuous map $f:\overline{M} \to \overline{N}$ which sends  a stratum of $\overline{M}$  into a stratum of $\overline{N}$; equivalently, the preimage of any stratum in $\overline{N}$ is a union of strata in $\overline{M}$. We shall also say that $f$ is stratum preserving.
\item We shall say that such a map is codimension preserving if,
in addition, for any stratum $S\subset \overline{N}$ 
$$
\textup{codim} \, f^{-1}(S) = \textup{codim} S.
$$
\item A stratified map $f:\overline{M} \to \overline{N}$ is smooth if it lifts to a smooth $b$-map
$\tilde{f}: \widetilde{M} \to \widetilde{N}$ between the respective resolutions. We shall also say
that $f$ is smoothly stratified.\\
We call a smoothly stratified, codimension preserving map a 
"smooth strongly stratum preserving map" or, equivalently,
a "smooth strongly stratified map". We will always work in this category.
\end{enumerate}
\end{defn}

Our setting in this section is as follows. Consider a smooth strongly stratum preserving map $f:\overline{M} \to \overline{N}$.
Consider the disc subbundle $\mathbb{B}^N \subset {}^{e}TN$, with fibres given by open discs of radius $\delta > 0$ with respect to 
the complete edge metric $\rho_N^{-2} g_N$. Here, we obviously assume $\delta>0$ to be a lower bound for the 
injectivity radius, which exists since $(N, \rho_N^{-2} g_N)$ is of bounded geometry. Let $\pi: \widetilde{f}^*(\mathbb{B}^N) \to \widetilde{M}$ be the pullback bundle 
with the usual associated bundle map $\mathbb{B}\widetilde{f}: \widetilde{f}^*(\mathbb{B}^N) \to \mathbb{B}^N$, 
induced by the right projection. Consider the exponential $\exp: \mathbb{B}^N|_N \to N$ with respect to $\rho_N^{-2} g_N$, 
evaluating tangent vectors at the end points of the corresponding geodesics at time $\delta$.
The map extends to $\exp: \mathbb{B}^N \to \widetilde{N}$, mapping 
$\mathbb{B}^N|_{\partial \widetilde{N}} \to \partial \widetilde{N}$.
These maps combine into a diagram as in Figure \ref{maps-diagram}.

\begin{figure}[h!]
\begin{center}
\begin{tikzpicture}


\draw[->] (5,0) -- (7,0);
\node at (4.5,0) {$\widetilde{M}$};
\node at (7.5,0) {$\widetilde{N}$};
\node at (6,-0.5) {$\widetilde{f}$};

\draw[->] (5.4,2) -- (6.8,2);
\node at (4.5,2) {$\widetilde{f}^*(\mathbb{B}^N)$};
\node at (7.5,2) {$\mathbb{B}^N$};
\node at (6,2.5) {$\mathbb{B}\widetilde{f}$};

\draw[->] (4.5,1.5) -- (4.5,0.5);
\draw[->] (7.5,1.5) -- (7.5,0.5);

\node at (4,1) {$\pi$};
\node at (7.1,1) {$\pi_N$};

\draw[->] (8,1.8) .. controls (8.5,1.5) and (8.5,0.5) .. (8,0.2);
\node at (9,1) {$\exp$};


\end{tikzpicture}
\end{center}
\caption{The pullback bundle $\widetilde{f}^*(\mathbb{B}^N)$.}
\label{maps-diagram}
\end{figure}

\begin{lem}\label{pullback-lemma} 
The differential $df$ in the open interior $M$ extends to
a fiberwise linear map between edge as well as wedge tangent bundles
$$
{}^{e}df: {}^{e}TM \to {}^{e}TN, \qquad {}^{w}df: {}^{w}TM \to {}^{w}TN.
$$
The corresponding pullbacks are defined by ${}^{e}df$ and 
${}^{w}df$ pointwise as follows.
\begin{enumerate}
\item for any $V_1, \cdots V_k \in {}^{e}TM$ and
$\omega \in C^\infty(\widetilde{N}, \Lambda^* ({}^{e}T^*N) \otimes \mathscr{E})$, 
\begin{align}\label{pullback-explicit-e}
{}^{e}f^* \omega (V_1, \cdots V_k) := \omega ({}^{e}df[V_1], \cdots {}^{e}df[V_k]) \circ f.
\end{align}
\item for any $V_1, \cdots V_k \in {}^{w}TM$ and
$\omega \in C^\infty(\widetilde{N}, \Lambda^* ({}^{w}T^*N) \otimes \mathscr{E})$, 
\begin{align}\label{pullback-explicit-w}
{}^{w}f^* \omega (V_1, \cdots V_k) := \omega ({}^{w}df[V_1], \cdots {}^{w}df[V_k]) \circ f.
\end{align}
\end{enumerate}
The pullbacks consequently act as follows
\begin{equation}\label{pullbacks}
\begin{split}
{}^{e}f^*: & C^\infty(\widetilde{N}, \Lambda^* ({}^{e}T^*N) \otimes \mathscr{E}) \to C^\infty(\widetilde{M}, \Lambda^* ({}^{e}T^*M) \otimes f^*\mathscr{E}), \\
{}^{w}f^*: & C^\infty(\widetilde{N}, \Lambda^* ({}^{w}T^*N) \otimes \mathscr{E}) \to C^\infty(\widetilde{M}, \Lambda^* ({}^{w}T^*M) \otimes f^*\mathscr{E}),
\end{split}
\end{equation}
where $\mathscr{E}$ is the Mishchenko bundle over $\overline{N}$, as in \eqref{mi-bundles}, pulled back to $\widetilde{N}$
by the blowdown map $\beta: \widetilde{N} \to \overline{N}$ that is associated with the resolution $\widetilde{N}$. Moreover, the pullbacks commute with the twisted differentials
$d_{\mathscr{E}}$ and $d_{\widetilde{f}^*\mathscr{E}}$ on the respective spaces of differential 
forms.
\end{lem}

\begin{proof}
The continuous extension of the differential to 
the edge tangent bundle, the so-called $e$-differential ${}^{e}df$, is obtained as follows. By assumption in Definition \ref{stratified-map-def}, $f$
preserves the iterated boundary fibrations structures and hence sends curves that are tangent to the fibres at the 
boundary of $M$ to curves that are tangent to the fibres at the boundary of $N$. Hence we obtain ${}^{e}df: {}^{e}TM \to {}^{e}TN$.

For the action of the differential between the wedge tangent bundles, 
recall that the total boundary functions $\rho_M \in C^\infty(\widetilde{M})$ and $\rho_N \in C^\infty(\widetilde{N})$ 
are nowhere vanishing in the open interior, and vanish to first
order at each boundary face of $\widetilde{M}$ and $\widetilde{N}$, respectively. 
We write for any $V \in T_pM$ with $p \in M$ in the open interior
\begin{equation}\label{d}
\begin{split}
df (V) &= \frac{\rho_N(f(p))}{\rho_M(p)} \cdot  \frac{\rho_M(p)}{\rho_N(f(p))}  \cdot df(V) \\
&= \frac{\rho_N(f(p))}{\rho_M(p)} \cdot \rho^{-1}_N(f(p)) \cdot df(\rho_M(p)V).
\end{split}
\end{equation}
If $V$ is now a section of ${}^{w}TM$, $\rho_M V$ extends to a section of ${}^{e}TM$, and 
$df(\rho_M V)$ extends to a section ${}^{e}df(\rho_M V)$ of ${}^{e}TN$. Consequently, 
$(\rho^{-1}_N\circ \widetilde{f}) \cdot df(V)$ extends to a section $\widetilde{f}^* \rho^{-1}_N \cdot {}^{e}df(\rho_M V)$ 
of ${}^{w}TN$. Now, since $f$ is smoothly stratified, we have that
$\widetilde{f}$ sends boundary hypersurfaces of $\widetilde{M}$ onto boundary hypersurfaces of $\widetilde{N}$.
Thus, as already explained in \cite[p. 298]{package}, $\rho^{-1}_M \cdot \widetilde{f}^* \rho_N$ is smooth and bounded 
(note that this is not the case for an arbitrary $b$-map), and we can 
write in view of \eqref{d} for any $V \in {}^{w}TM$
\begin{align}\label{wedge-diff-def}
{}^{w}df (V) = \frac{\widetilde{f}^* \rho_N}{\rho_M} \cdot \widetilde{f}^* \rho^{-1}_N \cdot {}^{e}df(\rho_M V) \in {}^{w}TN.
\end{align}
This provides a continuous extension of the differential to the wedge 
tangent bundles. If fact, all the boundary defining functions in \eqref{wedge-diff-def}
combine together to the constant function $1$, but the peculiar expression is needed
to explain that we indeed have a continuous extension to wedge vector fields in ${}^{w}TM$.
\medskip

The action of the pullbacks ${}^{e}f^*$ and ${}^{w}f^*$, 
as defined by ${}^{e}df$ and ${}^{w}df$
in \eqref{pullback-explicit-e} and \eqref{pullback-explicit-w} respectively, now follows.
It remains to prove that the pullbacks commute with the differentials. Let us write out the argument 
for the wedge case, the edge case is treated exactly in the same manner.
Consider any $V_0, \ldots, V_k \in {}^{w}TM$, $U_0, \ldots, U_k \in {}^{w}TN$ and
$\omega \in C^\infty(\widetilde{N}, \Lambda^k ({}^{w}T^*N) \otimes \mathscr{E})$. 
Locally, $\omega$ can be written as a sum of differential forms in 
$C^\infty(\widetilde{N}, \Lambda^k ({}^{w}T^*N))$ tensorized by parallel sections of
$\mathscr{E}$. For such terms we compute as follows.
\begin{align*}
d_{\mathscr{E}} \omega (U_0, \ldots U_k) &= \sum_{i=0}^\infty (-1)^i U_i \Bigl\{ 
\omega \left(U_0, \ldots, \widehat{U_i}, \ldots, U_k\right)\Bigr\}, \\
{}^{w}f^* d_{\mathscr{E}}\omega \left(V_0, \ldots, V_k\right) &= 
d_{\mathscr{E}} \omega \bigl({}^{w}df[V_0], \ldots, {}^{w}df[V_k]\bigr) \circ f \\
&= \sum_{i=0}^\infty (-1)^i \ {}^{w}df[V_i] \Bigl\{\omega \left({}^{w}df[V_0], \ldots, 
\widehat{{}^{w}df[V_i]}, \ldots, {}^{w}df[V_k]\right)\Bigr\} \circ f
\\ &= \sum_{i=0}^\infty (-1)^i V_i \Bigl\{\omega \left({}^{w}df[V_0], \ldots, 
\widehat{{}^{w}df[V_i]}, \ldots, {}^{w}df[V_k]\right) \circ f\Bigr\} 
\\ &= d_{\widetilde{f}^*\mathscr{E}} {}^{w}f^*\omega \bigl(V_0, \ldots, V_k\bigr).
\end{align*}
\end{proof}

Now, as it is well known even in the smooth closed case, the pullbacks in \eqref{pullbacks} do not necessarily induce bounded maps in $L^2$ and 
hence do not define morphisms of $\mathscr{N}\Gamma$-Hilbert complexes. 
This is precisely where the Hilsum-Skandalis type replacement of the pullback is needed. 
The next proposition is a central element of the construction. 

\begin{prop}\label{exp-f-prop} 
The differential $d(\exp \circ \, \mathbb{B}\widetilde{f})$ in the open interior 
of $\widetilde{f}^*(\mathbb{B}^N)$ extends to a well-defined surjection between edge and 
wedge tangent bundles
\begin{equation}\label{differentials}
\begin{split}
&{}^{e}d (\exp \circ \, \mathbb{B}\widetilde{f}): {}^{e}T \widetilde{f}^*(\mathbb{B}^N) \to {}^{e}T\widetilde{N}, \\
&{}^{w}d (\exp \circ \, \mathbb{B}\widetilde{f}): {}^{w}T \widetilde{f}^*(\mathbb{B}^N) \to {}^{w}T\widetilde{N}.
\end{split}
\end{equation}
Here, ${}^{e}T \widetilde{f}^*(\mathbb{B}^N)$ and ${}^{w}T \widetilde{f}^*(\mathbb{B}^N)$ are defined as follows: 
the disc subbundle $\mathbb{B}^N \subset {}^{e}TN$, as well as its pullback $\widetilde{f}^*(\mathbb{B}^N)$ 
are manifolds with corners and we consider the corresponding edge and wedge tangent bundles.
The pullbacks, defined by ${}^{e}d (\exp \circ \, \mathbb{B}\widetilde{f})$ and ${}^{w}d (\exp \circ \, \mathbb{B}\widetilde{f})$, act as follows 
\begin{equation}\label{exp-f-pullback}
\begin{split}
&{}^{e} (\exp \circ \, \mathbb{B}\widetilde{f})^*: C^\infty(\widetilde{N}, \Lambda^* ({}^{e}T^*\widetilde{N}) \otimes \mathscr{E})\to 
C^\infty(\widetilde{f}^*(\mathbb{B}^N), 
\Lambda^* ({}^{e}T^* \widetilde{f}^*(\mathbb{B}^N)\otimes \mathscr{E}'), \\
&{}^{w} (\exp \circ \, \mathbb{B}\widetilde{f})^*: C^\infty(\widetilde{N}, \Lambda^* ({}^{w}T^*\widetilde{N})\otimes \mathscr{E}) \to 
C^\infty(\widetilde{f}^*(\mathbb{B}^N), \Lambda^* ({}^{w}T^* \widetilde{f}^*(\mathbb{B}^N)\otimes \mathscr{E}').
\end{split}
\end{equation}
where $\mathscr{E}$ is the Mishchenko bundle as in Lemma \ref{pullback-lemma} and
$\mathscr{E}' := (\exp \circ \, \mathbb{B}\widetilde{f})^* \mathscr{E}$.
\end{prop}

\begin{proof} By  \cite[\S 4.2]{ALMP3}, the disc subbundle $\mathbb{B}^N$ 
is a manifold with iterated boundary fibration structure, where
each boundary hypersurface in $\mathbb{B}^N$ is the 
restriction of ${}^{e}TN$ to a boundary hypersurface of $\widetilde{N}$. Note that
the disc fibres of $\mathbb{B}^N$ are open by definition and hence the boundary 
of the discs is by definition not part of the boundary fibration structure. 
Since $f:\overline{M} \to \overline{N}$ is a smoothly stratified map, the induced bundle 
map $\mathbb{B}\widetilde{f}$ is again smoothly stratified.  
By \cite[Lemma 4.3]{ALMP3}  we know that $\exp$ is also smoothly stratified
and ${}^{e}d \exp$ is surjective (with respect to edge bundles). Surjectivity is not affected by the rescaling as
in \eqref{wedge-diff-def}. Hence the statement follows by 
exactly the same argument as in Lemma \ref{pullback-lemma}.
\end{proof}

 Below, it will become necessary to replace $\mathscr{E}' := 
 (\exp \circ \, \mathbb{B}\widetilde{f})^* \mathscr{E}$ by an isometric pullback bundle,
 where the isometry commutes with the pullback connections.
The notion of connections on $\mathscr{N}\Gamma$-Hilbert module bundles
is presented e.g. in \cite[\S 3]{Schick}. In fact, \cite[Definition 3.2]{Schick} introduces connections 
on $\mathscr{N}\Gamma$-Hilbert module bundles, \cite[Definition 3.4]{Schick} introduces pullback bundles 
and pullback connections. 

\begin{lemma}\label{retract-lemma} Consider the Mishchenko bundle 
$\mathscr{E} = \widetilde{N}_\Gamma \times_\Gamma \ell^2\Gamma$ with the canonical flat connection,
induced by the exterior derivative on $\widetilde{N}_\Gamma$. Consider the pullback bundles
$(\exp \circ \, \mathbb{B}\widetilde{f})^* \mathscr{E}$ and $(\widetilde{f} \circ \pi)^* \mathscr{E}$ with the
pullback connections, see Figure  \ref{maps-diagram} for the various maps. 
Then there exists an isometry of $\mathscr{N}\Gamma$-Hilbert module bundles,
commuting with the flat pullback connections
$$
\mathscr{J}: (\exp \circ \, \mathbb{B}\widetilde{f})^* \mathscr{E} \to  (\widetilde{f} \circ \pi)^* \mathscr{E}.
$$
\end{lemma}
\begin{proof}
Consider the continuous family 
$$
r_t: \mathbb{B}^N \to \mathbb{B}^N, \quad  \mathbb{B}^N_p  \ni v \mapsto tv \in \mathbb{B}^N_p,
$$
for $t \in [0,1]$. Clearly pullback by homotopic maps gives isomorphic bundles. We use now the homotopy and the flat connection to
explicitly construct a bundle isomorphism below, denoted by  $\mathscr{J}'$,  that 
preserves the given flat connections
$$
\mathscr{J}': r_0^* \exp^* \mathscr{E} \to r_1^* \exp^* \mathscr{E}.
$$
Consider the 
flat connection $\nabla$ on $\exp^* \mathscr{E}$, obtained as the pullback of the canonical
flat connection on the Mishchenko bundle $\mathscr{E}$. We claim that this isomorphism $\mathscr{J}'$
commutes with the pullback connections $r_0^*\nabla$ and $r_1^*\nabla$. Indeed, consider 
$$
r: [0,1] \times \mathbb{B}^N \to \mathbb{B}^N, \quad (t,v) \mapsto r_t(v).
$$
Consider the pullback bundle $r^*\exp^* \mathscr{E}$ and the flat pullback connection $r^*\nabla$ on 
that bundle. Clearly, we have the following restrictions
\begin{align*}
\begin{array}{lll}
&\Bigl. r^*\exp^* \mathscr{E} \Bigr|_{\{0\} \times \mathbb{B}^N} = r_0^* \exp^* \mathscr{E}, 
\quad &\Bigl. r^*\nabla  \Bigr|_{\{0\} \times \mathbb{B}^N} = r_0^*\nabla, \\
&\Bigl. r^*\exp^* \mathscr{E} \Bigr|_{\{1\} \times \mathbb{B}^N} = r_1^* \exp^* \mathscr{E}, 
\quad &\Bigl. r^*\nabla  \Bigr|_{\{1\} \times \mathbb{B}^N} = r_1^*\nabla. 
\end{array}
\end{align*}
The isomorphism $\mathscr{J}'$ may be constructed explicitly 
via parallel transport $P$ by $r^*\nabla$:
given any $(0,p) \in [0,1] \times \mathbb{B}^N$, consider the path $\gamma(t) := (t,p), t \in [0,1]$ and define 
$$
\mathscr{J}':= P_\gamma: r_0^* \exp^* \mathscr{E} \to r_1^* \exp^* \mathscr{E}.
$$
It is an isometry, since the metric in each fibre 
is given by the $\ell^2\Gamma$ inner product and parallel transport preserves the inner product.
The isomorphism commutes with the pullback connections $r_0^*\nabla$ and $r_1^*\nabla$, 
if for any continuous path $\gamma'$ between $p,q, \in \mathbb{B}^N$
\begin{align}\label{null-path}
P_\gamma \circ P_{\gamma'} = P_{\gamma'} \circ P_\gamma.
\end{align}
This is illustrated in the commutative diagram in Figure \ref{commuting-pullback-connections2}.
We illustrate the curves $\gamma$ and $\gamma'$ as well as the parallel transports along those in Figure \ref{commuting-pullback-connections}.

\begin{figure}[h!]
\centering
\begin{tikzpicture}

\draw[->] (5.4,0) -- (6.8,0);
\node at (4.5,0) {$\Bigl. r_0^* \exp^* \mathscr{E} \Bigr|_{q}$};
\node at (8,0) {$\Bigl. r_1^* \exp^* \mathscr{E} \Bigr|_{q}$};
\node at (6.3,-0.5) {$P_\gamma$};

\draw[->] (5.4,2) -- (6.8,2);
\node at (4.5,2) {$\Bigl. r_0^* \exp^* \mathscr{E} \Bigr|_{p}$};
\node at (8,2) {$\Bigl. r_1^* \exp^* \mathscr{E} \Bigr|_{p}$};
\node at (6.3,2.5) {$P_\gamma$};

\draw[->] (4,1.5) -- (4,0.5);
\draw[->] (8.2,1.5) -- (8.2,0.5);

\node at (3.2,1) {$P_{\gamma'}$};
\node at (9.1,1) {$P_{\gamma'}$};

\end{tikzpicture}

\caption{Parallel transport commutes along null-homotopic paths.}
\label{commuting-pullback-connections2}
\end{figure}

\begin{figure}[h!]
\centering
\begin{tikzpicture}

\node at (14.2,1.1) {\includegraphics[scale=0.5]{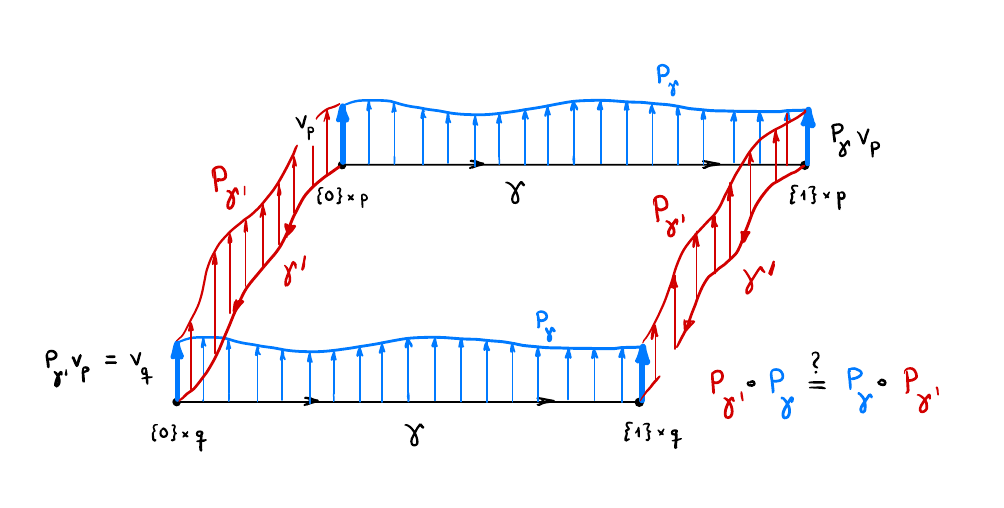}};

\end{tikzpicture}

\caption{Illustration of parallel transports across $[0,1] \times \mathbb{B}^N$.}
\label{commuting-pullback-connections}
\end{figure}

\noindent Now \eqref{null-path} holds 
since parallel transport for flat connections is trivial along null-homotopic paths. 
Let us explain why the closed path $F$ in Figure \ref{commuting-pullback-connections} is null-homotopic. 
The path $F$ is by definition traversed by travelling along the following sequence of paths in $[0,1] \times \mathbb{B}^N$
\begin{enumerate}
\item $(0,\gamma'(t)) \in \{0\} \times \mathbb{B}^N, \ t \in [0,1]$ \\
(we move along the bottom red curve at the left hand side in Fig. \ref{commuting-pullback-connections}),
\item $(t,q= \gamma'(1)) \in [0,1] \times \{q\}, \ t \in [0,1]$ \\ 
(we move along the bottom back curve in Fig. \ref{commuting-pullback-connections}), 
\item $(1,\gamma'(1-t)) \in \{1\} \times \mathbb{B}^N, \ t \in [0,1]$ \\ 
(we move backward along the bottom red curve at the RHS in Fig. \ref{commuting-pullback-connections}),
\item $(1-t,p= \gamma'(0)) \in [0,1] \times \{p\}, \ t \in [0,1]$ \\
(we move backward along the upper back curve in Fig. \ref{commuting-pullback-connections}).
\end{enumerate}
We replace that path $F$ by a family $F[s_1,s_2]$ of closed curves inside $[0,1] \times \mathbb{B}^N$
for any parameters $(s_1,s_2) \in [0,1]\times [0,1]$, namely we define $F[s_1,s_2]$ by
\begin{enumerate}
\item $(0,\gamma'(s_1 t)) \in \{0\} \times \mathbb{B}^N, \ t \in [0,1]$, \\
(we travel now only part of the way along the bottom red curve of the left hand side in Fig. \ref{commuting-pullback-connections})
\item $(s_2t,\gamma'(s_1)) \in [0,s_2] \times \{\gamma'(s_1)\}, \ t \in [0,1]$, \\
(we move parallel to the bottom back curve in Fig. \ref{commuting-pullback-connections}, and 
in fact again only part of the way), 
\item $(s_2,\gamma'(s_1(1-t))) \in \{s_2\} \times \mathbb{B}^N, \ t \in [0,1]$, \\
(we move now backwards parallel to the bottom red curve of the right hand side in Fig. \ref{commuting-pullback-connections}),
\item $(s_2(1-t),p= \gamma'(0)) \in [0,s_2] \times \{p\}, \ t \in [0,1]$ \\
(we move backwards along the upper back curve in Fig. \ref{commuting-pullback-connections}).
\end{enumerate}
The path $F= F(1,1)$ is null-homotopic by first deforming $F(1,1)$ to $F(0,1)$, 
and then deforming $F(0,1)$ to $F(0,0) \equiv \{(0,p)\}$. This proves \eqref{null-path}.
This proves that $\mathscr{J}'$ commutes with the pullback connections.
\medskip

We arrive at the following
sequence of bundle isomorphisms (note that $r_1 = \textup{id}$)
\begin{align*}
\exp^* \mathscr{E} = r_1^* \exp^* \mathscr{E} \stackrel{\mathscr{J}'}{\cong} r_0^* \exp^* \mathscr{E}
= (\exp \, \circ \, r_0)^*  \mathscr{E} = (\pi_N \circ r_0)^*  \mathscr{E} = \pi_N^* \mathscr{E}.
\end{align*}
From here we conclude, using $\pi_N \circ \mathbb{B}\, \widetilde{f} = \widetilde{f}  \circ \pi$
\begin{align*}
(\exp \circ \, \mathbb{B}\widetilde{f})^* \mathscr{E} = \mathbb{B}\widetilde{f}^* (\exp^* \mathscr{E})
\cong \mathbb{B}\widetilde{f}^* (\pi_N^* \mathscr{E}) = (\pi_N \circ \mathbb{B}\, \widetilde{f} )^* \mathscr{E} 
= (\widetilde{f} \circ \pi)^* \mathscr{E}.
\end{align*}
This constructs the isomorphism $\mathscr{J}$. Since the only non-trivial identification in 
$\mathscr{J}$ is the isomorphism $\mathscr{J}'$, it is an isometry and commutes with the pullback connections.
\end{proof}

We can now define the Hilsum-Skandalis replacement for $f^*$.
This requires a Thom form $\tau[{}^eTN]$ of the edge tangent bundle ${}^eTN$
as an additional ingredient. By definition, it is a closed differential form with compact support and 
of degree $\dim N$ on the total space of ${}^eTN$, such that its fibrewise integration gives the
constant function $1$. \medskip

Hence, in a local trivialization, the Thom form
is a family of top degree compactly supported differential forms along the fibres of 
${}^eTN$, varying smoothly in the base $\widetilde{N}$. Up to rescaling, 
we may assume that $\tau[{}^eTN]$ is compactly supported in the disc bundle $\mathbb{B}^N$.
Writing $p: {}^eTN \to \widetilde{N}$ for the bundle projection, 
we choose a connection that defines a decomposition into vertical and horizontal bundles
and thus we get an isomorphism
$$
{}^{e/w}T \mathbb{B}^N \cong p^* ({}^{e/w}TN) \oplus T^V \mathbb{B}^N,
$$ 
where the vertical component $T^V \mathbb{B}^N$ 
is the same both for the edge and wedge tangent bundles, since
the fibres in $\mathbb{B}^N$ are open discs.
As a differential form on $\mathbb{B}^N$, the Thom form 
$\tau[{}^eTN]$ has no horizontal components and hence can be viewed as a section of 
the exterior algebra of both ${}^{e}T^* \mathbb{B}^N$ and ${}^{w}T^* \mathbb{B}^N$, i.e.
$$
\tau[{}^eTN] \in C^\infty(\mathbb{B}^N,
\Lambda^* ({}^{e/w}T^* \mathbb{B}^N).
$$
Thus, also the pullbacks of $\tau[{}^eTN]$ by ${}^e(\mathbb{B}\widetilde{f})^*$ and 
${}^w(\mathbb{B}\widetilde{f})^*$ coincide, independent of whether we use
edge or wedge pullback, i.e.
\begin{align*}
(\mathbb{B}\widetilde{f})^*\tau[{}^eTN] :=
{}^e(\mathbb{B}\widetilde{f})^*\tau[{}^eTN] = {}^w(\mathbb{B}\widetilde{f})^*\tau[{}^eTN]
\in C^\infty(\widetilde{f}^*(\mathbb{B}^N), 
\Lambda^* ({}^{e/w}T^* \widetilde{f}^*(\mathbb{B}^N)).
\end{align*}
We can now define the Hilsum-Skandalis maps by a sequence of the following operations
in each degree $k$
\begin{align*}
&C^\infty\Bigl(\widetilde{N}, \Lambda^k ({}^{e/w}T^*N)\otimes \mathscr{E}\Bigr) 
\\ &\xrightarrow{{}^{e/w}(\exp \circ \, \mathbb{B}\widetilde{f})^*} \quad C^\infty\Bigl(\widetilde{f}^*(\mathbb{B}^N), 
\Lambda^k ({}^{e/w}T^* \widetilde{f}^*(\mathbb{B}^N))\otimes 
(\exp \circ \, \mathbb{B}\widetilde{f})^* \mathscr{E}\Bigr) \\
&\xrightarrow{\qquad \mathscr{J} \qquad } \quad C^\infty\Bigl(\widetilde{f}^*(\mathbb{B}^N), 
\Lambda^k ({}^{e/w}T^* \widetilde{f}^*(\mathbb{B}^N))\otimes (\widetilde{f} \circ \pi)^* \mathscr{E}\Bigr)
\\ &\xrightarrow{(\mathbb{B}\widetilde{f})^*\tau \bigl[{}^{e}T N \bigr] \wedge} 
 \quad C^\infty\Bigl(\widetilde{f}^*(\mathbb{B}^N), 
\Lambda^{k+\dim N} ({}^{e/w}T^* \widetilde{f}^*(\mathbb{B}^N))\otimes (\widetilde{f} \circ \pi)^* \mathscr{E}\Bigr)
\\ &\xrightarrow{\qquad \pi_* \qquad } \quad C^\infty\Bigl(\widetilde{M}, 
\Lambda^{k} ({}^{e/w}T^* M)\otimes \widetilde{f}^* \mathscr{E}\Bigr).
\end{align*}

\begin{thm}\label{HS-thm} Let $f: \overline{M} \to \overline{N}$ be smooth, strongly stratum preserving. 
Consider thecThom form $\tau[{}^{e}T N]$.
Then the Hilsum-Skandalis maps 
\begin{equation}\label{HS}
\begin{split}
&{}^{e}HS(f): C^\infty(\widetilde{N}, \Lambda^* ({}^{e}T^*N)\otimes \mathscr{E}) \to C^\infty(\widetilde{M}, \Lambda^* ({}^{e}T^*M)\otimes \widetilde{f}^*\mathscr{E}), \\ 
&{}^{e}HS(f) \omega  := \pi_* \Bigl((\mathbb{B}\widetilde{f})^*\tau \bigl[{}^{e}T N \bigr] \wedge \mathscr{J} 
\circ {}^{e}(\exp \circ \, \mathbb{B}\widetilde{f})^* \omega \Bigr) \\
&\qquad \qquad {}^{w}HS(f): C^\infty(\widetilde{N}, \Lambda^* ({}^{w}T^*N)\otimes \mathscr{E}) \to C^\infty(\widetilde{M}, \Lambda^* ({}^{w}T^*M)\otimes \widetilde{f}^*\mathscr{E}), \\ 
&\qquad \qquad {}^{w}HS(f) \omega  := \pi_* \Bigl((\mathbb{B}\widetilde{f})^*\tau\bigl[{}^{e}T N\bigr] \wedge \mathscr{J} \circ {}^{w}(\exp \circ \, \mathbb{B}\widetilde{f})^* \omega \Bigr),
\end{split}
\end{equation}
are well-defined and extend to bounded maps on the corresponding $L^2$ completions
\begin{equation}\label{HSL2}
\begin{split}
&{}^{e}HS(f): L^2(\widetilde{N}, \Lambda^* ({}^{e}T^*N)\otimes \mathscr{E}, \rho_N^{-2} g_{\mathscr{E}}) 
\to L^2(\widetilde{M}, \Lambda^* ({}^{e}T^*M)\otimes \widetilde{f}^*\mathscr{E}, \rho_M^{-2}g_{\mathscr{E}}), \\
&{}^{w}HS(f): L^2(\widetilde{N}, \Lambda^* ({}^{w}T^*N) \otimes \mathscr{E}, g_{\mathscr{E}}) 
\to L^2(\widetilde{M}, \Lambda^* ({}^{w}T^*M)\otimes \widetilde{f}^*\mathscr{E}, g_{\mathscr{E}}).
\end{split}
\end{equation}
The maps commute with the twisted de Rham differentials 
$d_{\mathscr{E}}$ and $d_{\widetilde{f}^*\mathscr{E}}$ 
on smooth sections. The metrics $g_{\mathscr{E}}$ are induced by the wedge metrics 
$g_N,g_M$ together with the 
$\ell^2\Gamma$ inner product along the fibers of $\mathscr{E}$.  
\end{thm}

\begin{proof}
The maps are well-defined by Proposition \ref{exp-f-prop}. 
We now prove \eqref{HSL2}. Note first that $(M, \rho_M^{-2}g_{M})$ and $(N, \rho_N^{-2}g_{N})$
are of bounded geometry. Hence, boundedness of ${}^{e}HS(f)$ is discussed in detail in
\cite[Proposition 4.5]{Spess}, who extended \cite{Hilsum-Skandalis} to non-compact 
Riemannian manifolds of bounded geometry. Note here, that the construction of ${}^{e}HS(f)$ differs from
$T_f$ in \cite[(112), (147)]{Spess} only by the isomorphism $\mathscr{J}$.

We now establish boundedness of ${}^{w}HS(f)$: note that $(\exp \circ \, \mathbb{B}\widetilde{f})$
and $\pi$ are both smooth, strongly stratum preserving. Consequently, writing $\rho_N$ and
$\rho_{\widetilde{f}^*(\mathbb{B}^N)}$ for the total boundary defining functions on $N$ and 
$\widetilde{f}^*(\mathbb{B}^N)$, respectively, we conclude that the quotients
$$
\frac{(\exp \circ \, \mathbb{B}\widetilde{f})^*\rho_N}{\rho_{\widetilde{f}^*(\mathbb{B}^N)}}, 
\quad \frac{\rho_{\widetilde{f}^*(\mathbb{B}^N)}}{\pi^*\rho_M}, 
$$
are uniformly bounded. Given now 
$\omega \in C^\infty(\widetilde{N}, \Lambda^k ({}^w T^*N)\otimes \mathscr{E})$
of any fixed degree $k$, we note that 
$\rho_N^{-k} \omega \in C^\infty(N, \Lambda^k ({}^e T^*N)\otimes \mathscr{E})$.
Therefore, we may compute
\begin{align*}
{}^{w}HS(f) \omega &= \pi_* \Bigl((\mathbb{B}\widetilde{f})^* \tau\bigl[{}^{e}TN\bigr]
\wedge \mathscr{J} \circ 
{}^w(\exp \circ \, \mathbb{B}\widetilde{f})^* (\rho_N^{k} \rho_N^{-k} \omega) \Bigr) \\
&= \pi_* \Bigl((\mathbb{B}\widetilde{f})^* \tau\bigl[{}^{e}TN\bigr] \wedge \mathscr{J} \circ 
 (\exp \circ \, \mathbb{B}\widetilde{f})^* \rho_N^k \cdot  {}^e(\exp \circ \, \mathbb{B}\widetilde{f})^* (\rho_N^{-k} \omega) \Bigr) \\
 &= \rho^k_M \cdot \pi_* \Bigl( \alpha 
 \cdot (\mathbb{B}\widetilde{f})^* \tau\bigl[{}^{e}TN\bigr] \wedge \mathscr{J} \circ 
 {}^e(\exp \circ \, \mathbb{B}\widetilde{f})^* (\rho_N^{-k} \omega) \Bigr), \\
&\textup{where} \ \alpha := (\exp \circ \, \mathbb{B}\widetilde{f})^*\rho^k_N
\cdot \pi^*\rho^{-k}_M.
\end{align*}
The presence of the bounded factor $\alpha$
does not affect the argument in \cite[Proposition 4.5]{Spess} and hence $\rho^{-k}_M \cdot {}^{w}HS(f) \rho_N^k$
is again bounded on the $L^2$ completions of $k$-th degree twisted differential forms 
with respect to $\rho_N^{-2} g_{\mathscr{E}}$ and $\rho_M^{-2} g_{\mathscr{E}}$.
This is equivalent to boundedness of ${}^{w}HS(f)$ in \eqref{HSL2} as claimed.
\medskip

The last claim that ${}^{e/w}HS(f)$ commutes with the de Rham differentials 
follows, as already observed in \cite[Proposition 9.3 (i)]{package} and \cite{Alb-survey}, 
from the fact that both maps are compositions of the isometry $\mathscr{J}$, 
pullback and exterior product by a closed form. The pullbacks commute
with the twisted differentials, as noted in Lemma \ref{pullback-lemma}. 
The isometry $\mathscr{J}$ commutes with the pullback connections by 
Lemma \ref{retract-lemma} and hence ${}^{e/w}HS(f)$ commute with the twisted differentials.
\end{proof}

\subsection{Hilsum-Skandalis maps for homotopic maps} 

Consider a continuous family $f_s: \overline{M} \to \overline{N}$ of smooth strongly stratum preserving maps.
Consider the pullback bundles $\widetilde{f}_s^*(\mathbb{B}^N)$. Homotopic maps induce vector bundle isomorphisms,
and hence there exists a continuous family $A_s: \widetilde{f}_0^*(\mathbb{B}^N) \to \widetilde{f}_s^*(\mathbb{B}^N)$
of vector bundle isomorphisms, leading to the commutative diagram explained in  Fig. \ref{pullbacks-figure} below

\begin{figure}[h!]
\begin{center}
\begin{tikzpicture}


\draw[->] (2,2) -- (3.5,2);
\draw[->] (1.5,1.5) -- (4,0);
\node at (1,2) {$\widetilde{f}_0^*(\mathbb{B}^N)$};
\node at (2.5,2.4) {$A_s$};

\draw[->] (5,0) -- (7,0);
\node at (4.5,0) {$\widetilde{M}$};
\node at (7.5,0) {$\widetilde{N}$};
\node at (6,-0.5) {$\widetilde{f}_s$};

\draw[->] (5.4,2) -- (6.8,2);
\node at (4.5,2) {$\widetilde{f}_s^*(\mathbb{B}^N)$};
\node at (7.5,2) {$\mathbb{B}^N$};
\node at (6,2.5) {$\mathbb{B}\widetilde{f}_s$};

\draw[->] (4.5,1.5) -- (4.5,0.5);
\draw[->] (7.5,1.5) -- (7.5,0.5);

\node at (4,1) {$\pi_s$};
\node at (1.5,1) {$\pi$};

\draw[->] (8,1.8) .. controls (8.5,1.5) and (8.5,0.5) .. (8,0.2);
\node at (9,1) {$\exp$};

\end{tikzpicture}
\end{center}
\caption{The pullback bundles $\widetilde{f}_s^*(\mathbb{B}^N)$ and $\widetilde{f}_0^*(\mathbb{B}^N)$.}
\label{pullbacks-figure}
\end{figure}

\noindent Similarly, there exists a continuous family $\mathscr{J}_s: \widetilde{f}_s^* \mathscr{E} \to \widetilde{f}_0^* \mathscr{E}$ 
of bundle isomorphisms, which by the same argument as in Lemma \ref{retract-lemma} are isometries of $\mathscr{N}\Gamma$-Hilbert module bundles,
commuting with the pullback connections. Motivated by Theorem \ref{HS-thm} we set (we change the construction slightly)
\begin{equation}\label{HS-prime}
\begin{split}
&{}^{w}HS(f_s)': C^\infty(\widetilde{N}, \Lambda^\ell ({}^{w}T^*N) \otimes \mathscr{E}) \to 
C^\infty(\widetilde{M}, \Lambda^\ell ({}^{w}T^*M)\otimes \widetilde{f}_0^*\mathscr{E}), \\ 
&{}^{w}HS(f_s)' \omega  := (-1)^{\ell} \, {}^{w}\pi_* \Bigl( \tau\bigl[\widetilde{f}_0^*(\mathbb{B}^N)\bigr] 
\wedge \mathscr{J}_s \circ \mathscr{J} \circ {}^{w}(\exp \circ \, \mathbb{B}\widetilde{f}_s \circ A_s)^* \omega \Bigr),
\end{split}
\end{equation}
where $\mathscr{J}$ is constructed in Lemma \ref{retract-lemma}.
As before, ${}^{w}HS(f_s)'$ extends to a bounded map between the $L^2$-completions. The following
proposition explains how these maps define a chain homotopy. 

\begin{prop}\label{chain-homotopy1} Write 
$p_s:= \exp \circ \, \mathbb{B}\widetilde{f}_s \circ A_s: \widetilde{f}_0^*(\mathbb{B}^N) \to \widetilde{N}$ and set
$$
p: \widetilde{f}_0^*(\mathbb{B}^N) \times [0,1] \to \widetilde{N}, \qquad (v,s) \mapsto p_s(v).
$$
Then ${}^{w}HS(f_1)' - {}^{w}HS(f_0)' = dK + Kd$ with $K$ bounded in $L^2$ and defined by
\begin{equation}
\begin{split}
&K: C^\infty(\widetilde{N}, \Lambda^* ({}^{w}T^*N)\otimes \mathscr{E}) \to C^\infty(\widetilde{M}, \Lambda^{*-1} ({}^{w}T^*M)\otimes \widetilde{f}_0^* \mathscr{E}), \\ 
&K \omega  := (-1)^{\ell} \, {}^{w}\pi_* \Bigl( \tau\bigl[\widetilde{f}_0^*(\mathbb{B}^N)\bigr]  \wedge 
\mathscr{J}_s \circ \mathscr{J} \circ \int_0^1 \iota_{\partial_s} \bigl( {}^{w}p^* \omega \bigr) \Bigr),
\end{split}
\end{equation}
\end{prop}

\begin{proof}
Consider $\omega \in C^\infty(\widetilde{N}, \Lambda^\ell ({}^{w}T^*N)\otimes \mathscr{E})$. The statement follows by plugging the next computation
into \eqref{HS-prime}. We change the notation of our differentials $d_{\mathscr{E}}$ by changing the lower index to specify the underlying space and obtain
\begin{align*}
d_{\widetilde{f}_0^*(\mathbb{B}^N)} &\int_0^1 \iota_{\partial_s} \bigl( {}^{w}p^* \omega \bigr)
= \int_0^1 \iota_{\partial_s} d_{\widetilde{f}_0^*(\mathbb{B}^N)} \bigl( {}^{w}p^* \omega \bigr)
\\ =  \, &\int_0^1 \iota_{\partial_s} \Bigl( d_{\widetilde{f}_0^*(\mathbb{B}^N)\times [0,1]} \bigl( {}^{w}p^* \omega \bigr)
+ (-1)^{\ell+1} \frac{\partial \bigl( {}^{w}p^* \omega \bigr)}{\partial s} \wedge ds \Bigr)
\\ =  \, &\int_0^1 \iota_{\partial_s} \Bigl( {}^{w}p^* \bigl( d_{\widetilde{N}}\omega \bigr)
+ (-1)^{\ell+1} \frac{\partial \bigl( {}^{w}p^* \omega \bigr)}{\partial s} \wedge ds \Bigr)
\\ =  \, &\int_0^1 \iota_{\partial_s} \Bigl( {}^{w}p^* \bigl( d_{\widetilde{N}}\omega \bigr)\Bigr)
+ {}^{w}p_1^* \omega - {}^{w}p_0^* \omega.
\end{align*}
\end{proof}

\section{Stability of $L^2$-Betti numbers and Novikov-Shubin invariants}\label{stability-section}

Let us first recall the notion of homotopy equivalences between $\mathscr{N}\Gamma$-Hilbert complexes.
We refer to \cite[\S 1 and \S 2]{Lueck-book} and \cite[\S 4]{Gromov-Shubin} for more details. 
We continue in the notation fixed in Definition \ref{gamma-notions2}.

\begin{defn}\label{gamma-notions4} Let $\Gamma$ be a finitely generated discrete group
and recall Definition \ref{gamma-notions2}.
\begin{enumerate}
\item A homotopy between
two morphisms of $\mathscr{N}\Gamma$-Hilbert complexes $f,h:H_{\bullet}\rightarrow K_{\bullet}$ 
is a sequence of bounded linear operators $T_k :H_k\rightarrow 
K_{k-1}$ commuting with the $\Gamma$-action such that 
$$
f_k-h_k = T_{k+1}\circ d_k+d_{k-1}\circ T_k \ \textup{on} \ \dom(d_k) \subset H_k.
$$ 

\item Two $\mathscr{N}\Gamma$-Hilbert complexes $H_{\bullet}$ and $K_{\bullet}$ are said to be 
homotopy equivalent if there are morphisms $f:H_{\bullet}\rightarrow K_{\bullet}$ and 
$h:H_{\bullet}\rightarrow K_{\bullet}$ such that $f\circ h$ and $h\circ f$ are homotopic 
to the identity morphisms of $H_{\bullet}$ and $K_{\bullet}$ respectively. 
\end{enumerate}
\end{defn}

We now introduce homotopy equivalences
in the smoothly stratified setting.

\begin{defn}\label{hom-equiv-def} Consider two compact smoo\-thly stratified spaces $\overline{M}$ and $\overline{N}$.
We say that $\overline{M}$ and $\overline{N}$ are 
"smoothly stratified codimension-preserving homotopy equivalent"
if there exist two smooth strongly stratum preserving maps $f: \overline{M} \to \overline{N}$
and $h: \overline{N} \to \overline{M}$, such that $f\circ h \sim \textup{id}_{\overline{N}}$ and 
$h\circ f \sim \textup{id}_{\overline{M}}$ with homotopies given by continuous families of 
smooth strongly stratum preserving maps.
\end{defn}

\noindent
For more on homotopy equivalences in the smoothly stratified setting we refer the reader to \cite{ALMP3},
where, in particular, a result of approximation of smoothly stratified codimension-preserving 
homotopy equivalences by smooth ones is established.

\smallskip
\smallskip
\noindent
We show that an equivalence as in Definition \ref{hom-equiv-def} induces
a chain homotopy equivalence of minimal and maximal Hilbert complexes for 
$\overline{M}_\Gamma$ and $\overline{N}_\Gamma$. 

\begin{cor}\label{chain-corollary}
Let $\overline{M}$ and $\overline{N}$ be compact  smoothly stratified,
strongly stratum preserving homotopy equivalent in the sense of Definition \ref{hom-equiv-def}.
Let $\overline{N}_\Gamma$ be a Galois $\Gamma$-covering on $\overline{N}$.
Let $\mathcal{E}_{\overline{N}}$ be the associated Mishchenko bundle over $\overline{N}$.
Then the $\mathscr{N}\Gamma$-Hilbert complexes 
$$
(\dom_{ \max}(M, f^* \mathscr{E}_{\overline{N}}), d_{f^* \mathscr{E}_{\overline{N}}}) \ \textup{and} \
(\dom_{ \max}(N, \mathscr{E}_{\overline{N}}), d_{\mathscr{E}_{\overline{N}}}).
$$ 
are chain homotopy equivalent.\\
\end{cor}

\begin{proof} 
Consider the two smooth strongly stratum preserving maps $f: \overline{M} \to \overline{N}$
and $h: \overline{N} \to \overline{M}$ that define the smoothly stratified codimension-preserving homotopy equivalence between 
$\overline{M}$ and $\overline{N}$. In particular, $f\circ h \sim \textup{id}_{\overline{N}}$ 
with the homotopy given by a continuous family $u_s: \overline{N} \to \overline{N}, s\in [0,1]$, of 
smooth strongly stratum preserving maps. Similarly, $h\circ f \sim \textup{id}_{\overline{M}}$ 
with the homotopy given by a continuous family $v_s: \overline{M} \to \overline{M}, s\in [0,1]$, of 
smooth strongly stratum preserving maps. We have
$$
u_1= f\circ h, u_0 = \textup{id}_{\overline{N}}, \qquad 
v_1= h\circ f, v_0 = \textup{id}_{\overline{M}}.
$$
Consider the Hilsum-Skandalis replacements ${}^{w}HS(u_s)'$ and ${}^{w}HS(v_s)'$ of the homotopies.
By Theorem \ref{HS-thm} these maps are bounded on the $L^2$ completions, commute with the twisted differential on smooth sections and hence act between $\Omega^{*}_{\max}(M,\mathscr{E})$ and 
$\Omega^{*}_{\max}(N,\mathscr{E})$. Those spaces form core subdomains for 
$\dom^*_{\max}(M, \mathscr{E})$and $\dom^*_{\max}(N, \mathscr{E})$, respectively. 
Hence we obtain bounded maps between their graph closures, the maximal Hilbert complexes
\begin{align*}
&{}^{w}HS(u_s)': (\dom_{\max}(N, \mathscr{E}), d_{\mathscr{E}}) \to (\dom_{\max}(N, \mathscr{E}), d_{\mathscr{E}}), \\
&{}^{w}HS(v_s)': (\dom_{\max}(M, \mathscr{E}), d_{\mathscr{E}}) \to (\dom_{\max}(M, \mathscr{E}), d_{\mathscr{E}}).
\end{align*}
Moreover, by Proposition \ref{chain-homotopy1}, we find for appropriate $K_u, K_v$
\begin{equation}\label{HK}
\begin{split}
&{}^{w}HS(f\circ h)' - {}^{w}HS(\textup{id}_{\overline{N}})'  \equiv {}^{w}HS(u_1)' - {}^{w}HS(u_0)' = d_NK_u + K_ud_N, \\
&{}^{w}HS(h \circ f)' - {}^{w}HS(\textup{id}_{\overline{M}})' \equiv {}^{w}HS(v_1)' - {}^{w}HS(v_0)' = d_MK_v + K_vd_M.
\end{split}
\end{equation}
Note that $K_u$ and $K_v$ do not commute with the differential and hence their boundedness in the $L^2$ spaces, as asserted in 
Proposition \ref{chain-homotopy1}, does not yet imply that they preserve the minimal and maximal domains. However, their relation 
(see \eqref{HK}) with $HS(u_s)'$ and $HS(v_s)'$, which do preserve the domains, imply 
\begin{equation}\label{K}
\begin{split}
&K_u: \dom^*_{\max}(N, \mathscr{E}) \to \dom^{*-1}_{\max}(N, \mathscr{E}),\\
&K_v: \dom^*_{\max}(M, \mathscr{E}) \to \dom^{*-1}_{\max}(M,  \mathscr{E}).
\end{split}
\end{equation}
Indeed we know that $K_u:L^2 \Omega^* (N, \mathscr{E})\rightarrow L^2 \Omega^* (N, \mathscr{E})$ is bounded and, by \eqref{HK}, we have $$d_NK_u={}^{w}HS(u_1)' - {}^{w}HS(u_0)'- K_ud_N.$$ Note that the above equality implies immediately that 
$$
K_u(\Omega^{*}_{\max}(N,\mathscr{E}))\subset \Omega^{*}_{\max}(N,\mathscr{E})).
$$ 
Since, in addition, $K_u:L^2 \Omega^* (N, \mathscr{E})\rightarrow L^2 \Omega^* (N, \mathscr{E})$ is bounded and $\Omega^{*}_{\max}(N,\mathscr{E})$ is a core subdomain for $\dom^*_{\max}(N, \mathscr{E})$, we can conclude that $K_u$ preserves $\dom^*_{\max}(N, \mathscr{E})$. 
Note that we can not argue in the same way for the minimal domain, since $K_u$
does not preserve compact support. \medskip

If we had ${}^{w}HS(\textup{id}_{\overline{N}})' = \textup{Id}$ and ${}^{w}HS(\textup{id}_{\overline{M}})' = \textup{Id}$,
this would already mean that the $\mathscr{N}\Gamma$-Hilbert complexes $(\dom_{\max}(M, \mathscr{E}), d_{\mathscr{E}})$ and 
$(\dom_{\max}(N), d_{\mathscr{E}})$ are chain homotopy equivalent
as in Definition \ref{gamma-notions4}. This would finish the proof. However, this is not the case. In order to 
compute $({}^{w}HS(\textup{id}_{\overline{N}})' - \textup{Id})$ as well as $({}^{w}HS(\textup{id}_{\overline{M}})' - \textup{Id})$,
we consider the exponential maps as above (recall that the disc fibres in the bundles $ \mathbb{B}^N$ and $\mathbb{B}^M$ are 
of radius $\delta$, some lower bound for the injectivity radius)
\begin{align*}
&\exp^N: \mathbb{B}^N |_N \to N, \qquad V \mapsto \gamma_V(\delta), \\
&\exp^M: \mathbb{B}^M |_M \to M \qquad V \mapsto \gamma_V(\delta),
\end{align*}
where $\gamma_V$ denotes the geodesic with respect to the complete edge metrics on either $N$ or $M$. 
Instead of evaluating the geodesics at $s=\delta$, we evaluate them at any $s \in [0,\delta]$ and define
\begin{align*}
&\exp_s^N: \mathbb{B}^N |_N \to N, \qquad V \mapsto \gamma_V(s), \\
&\exp_s^M: \mathbb{B}^M |_M \to M \qquad V \mapsto \gamma_V(s),
\end{align*}
Applying Proposition \ref{chain-homotopy1} to $p_s$, that is defined in terms of the extensions
$\exp^N_s: \mathbb{B}^N \to \widetilde{N}$
and $\exp^M_s: \mathbb{B}^M \to \widetilde{M}$, gives 
\begin{equation}\label{HK2}
\begin{split}
&{}^{w}HS(\textup{id}_{\overline{N}})' - \textup{Id} = d_{\mathscr{E}}K'_u + K'_ud_{\mathscr{E}}, \qquad K'_u: 
\dom^*_{\max}(N, \mathscr{E}) \to \dom^{*-1}_{ \max}(N, \mathscr{E}),\\
&{}^{w}HS(\textup{id}_{\overline{M}})' - \textup{Id} = d_{\mathscr{E}}K'_v + K'_vd_{\mathscr{E}} \qquad K'_v: 
\dom^*_{ \max}(M, \mathscr{E}) \to \dom^{*-1}_{ \max}(M, \mathscr{E}).
\end{split}
\end{equation}
Combining \eqref{HK}, \eqref{K} and \eqref{HK2} proves the statement. 
\end{proof}

\subsection{Proof of stability}

Our proof is now a consequence of the abstract stability result in 
\cite[Proposition 4.1]{Gromov-Shubin}, see also \cite[Theorem 2.19]{Lueck-book}, which we now state for convenience.

\begin{thm}\label{Gromov-Shubin}
Consider two $\mathscr{N}\Gamma$ Hilbert complexes $(\dom,d)$
and $(\dom',d')$. If they are chain homotopy equivalent as $\mathscr{N}\Gamma$ Hilbert complexes, then 
their spectral density functions $F_*(\lambda, (\dom,d))$ and $F_*(\lambda, (\dom',d'))$ 
are dilatationally equivalent, cf. \cite[Definition 2.7]{Lueck-book}, and in particular we have the following.
\begin{enumerate}
\item The Betti numbers $b^*_{(2),\Gamma}(\dom,d)$ and $b^*_{(2),\Gamma}(\dom',d')$ are equal.
\item $(\dom,d)$ is Fredholm iff $(\dom',d')$ is so and the Novikov-Shubin invariants $\alpha_*(\dom,d)$ and $\alpha_*(\dom',d')$ are equal.

\end{enumerate}
\end{thm}

We can now prove our second main result, Theorem \ref{main2}.

\begin{proof}[Proof of Theorem \ref{main2}] 

\textbf{Step 1)} Using (in both lines below) Proposition \ref{minmax-thm} for the first equality and 
Corollary \ref{chain-corollary} for the second equality (this is the only place where the
Hilsum-Skandalis replacement is used), we know that 
\begin{align*}
b^*_{(2),\mathrm{max}}(\overline{N}_{\Gamma})
&= b^*_{(2),\mathrm{max}}(\overline{N},\mathcal{E}_{\overline{N}_{\Gamma}})
= b^*_{(2),\mathrm{max}}(\overline{M},f^*\mathcal{E}_{\overline{N}_{\Gamma}}), \\
\alpha_{*,\mathrm{max}}(\overline{N}_{\Gamma})
&= \alpha_{*,\mathrm{max}}(\overline{N},\mathcal{E}_{\overline{N}_{\Gamma}})
= \alpha_{*,\mathrm{max}}(\overline{M},f^*\mathcal{E}_{\overline{N}_{\Gamma}}).
\end{align*}

\noindent \textbf{Step 2)} We claim that $f^*\mathcal{E}_{\overline{N}_{\Gamma}} \cong \mathcal{E}_{f^*\overline{N}_{\Gamma}}$ and that the 
isomorphism commutes with the natural flat pullback connections. The proof amounts to understanding the isomorphism explicitly. 
Consider an atlas $\{U_\alpha\}_{\alpha}$ on $\overline{N}$, consisting of trivializing neighborhoods for the Galois $\Gamma$-covering 
$p:\overline{N}_{\Gamma} \to \overline{N}$. The covering comes with local trivializations 
$$
\phi_\alpha: \Bigl. \overline{N}_{\Gamma} \Bigr|_{U_\alpha} \to U_\alpha \times \Gamma, \quad x \mapsto (p(x),\phi_\alpha(x)), 
$$
and transition functions $\phi_{\alpha, \beta}: U_\alpha \cap U_\beta \to \Gamma$. Let us recall the construction of the associated flat bundle 
$\mathcal{E}_{\overline{N}_{\Gamma}} = \overline{N}_{\Gamma} \times_\Gamma \ell^2\Gamma$ in order to fix notation. The 
Mishchenko bundle is defined as a set of equivalence classes 
$$
\mathcal{E}_{\overline{N}_{\Gamma}} := \{ [(x,v)]  \subset \overline{N}_{\Gamma} \times \ell^2\Gamma
\mid  \forall_{\gamma \in \Gamma}: (x,v) \sim (\gamma \cdot x, \gamma^{-1} \cdot v)\},
$$
where $\Gamma$ acts naturally on $\overline{N}_{\Gamma}$ and $\ell^2\Gamma$.
Its bundle structure comes from the local trivializations
(the transition functions are given by the action of $\phi_{\alpha, \beta}$ on $\ell^2\Gamma$)
$$
\Phi_\alpha: \Bigl. \mathcal{E}_{\overline{N}_{\Gamma}} \Bigr|_{U_\alpha} \to U_\alpha \times \ell^2\Gamma, \quad 
[(x,v)] \mapsto (p(x),\phi_\alpha(x) \cdot v).
$$
We now consider the pullbacks of these objects by $f$: 
the pullback $f^*\overline{N}_{\Gamma}$ is a Galois $\Gamma$-covering over $\overline{M}$,
and the pullback $f^*\mathcal{E}_{\overline{N}_{\Gamma}}$ as a bundle of Hilbert modules. Both are defined as a set as follows
\begin{align*}
&f^*\overline{N}_{\Gamma} := \left\{(q,x) \in \overline{M} \times \Bigl. \overline{N}_{\Gamma} \mid 
f(q) = p(x) \right\}, \\
&f^*\mathcal{E}_{\overline{N}_{\Gamma}} :=  \left\{(q; [(x,v)]) \in \overline{M} \times \Bigl. \mathcal{E}_{\overline{N}_{\Gamma}} 
\mid f(q) = p(x) \right\}.
\end{align*}
Their covering, respectively the bundle structures are given in terms of the atlas $\{f^{-1}(U_\alpha\}_{\alpha})$ on $\overline{M}$
of trivializing neighborhoods and local trivializations
\begin{align*}
&f^* \phi_\alpha: \Bigl. f^*\overline{N}_{\Gamma} \Bigr|_{f^{-1}(U_\alpha)} \to f^{-1}(U_\alpha) \times \Gamma, \quad (q,x) \mapsto (q,\phi_\alpha(x)), \\
&f^*\Phi_\alpha: \Bigl. f^*\mathcal{E}_{\overline{N}_{\Gamma}} \Bigr|_{f^{-1}(U_\alpha)} \to f^{-1}(U_\alpha) \times \ell^2\Gamma, 
\quad (q;[(x,v)]) \mapsto (q,\phi_\alpha(x)\cdot v), 
\end{align*}
and transition functions given in both cases by $f^*\phi_{\alpha, \beta}: f^{-1}(U_\alpha) \cap f^{-1}(U_\beta) \to \Gamma$.
Our task is to compare $f^*\mathcal{E}_{\overline{N}_{\Gamma}}$ to the bundle 
$\mathcal{E}_{f^*\overline{N}_{\Gamma}} := f^*\overline{N}_{\Gamma} \times_\Gamma \ell^2 \Gamma$. The latter
bundle is defined as a set of equivalence classes
$$
\mathcal{E}_{f^*\overline{N}_{\Gamma}} := \Bigl\{ \Bigl[\bigl((q,x); v\bigr)\Bigr]  \subset f^*\overline{N}_{\Gamma} \times \ell^2\Gamma
\mid  \forall_{\gamma \in \Gamma}: \bigl((q,x); v\bigr) \sim \bigl((q,\gamma \cdot x), \gamma^{-1} \cdot v\bigr) \Bigr\}.
$$
Its bundle structure is given by the local trivializations
\begin{align*}
\Psi_\alpha: \Bigl. \mathcal{E}_{f^*\overline{N}_{\Gamma}} \Bigr|_{f^{-1}(U_\alpha)} \to f^{-1}(U_\alpha) \times \ell^2\Gamma, 
\quad \Bigl[\bigl((q,x); v\bigr)\Bigr] \mapsto (q,\phi_\alpha(x)\cdot v), 
\end{align*}
with transition functions given as before for $f^*\mathcal{E}_{\overline{N}_{\Gamma}}$ 
by $f^*\phi_{\alpha, \beta}$. We can now define an isomorphism 
$f^*\mathcal{E}_{\overline{N}_{\Gamma}} \cong \mathcal{E}_{f^*\overline{N}_{\Gamma}}$, namely
$$
I: \mathcal{E}_{f^*\overline{N}_{\Gamma}} \to f^*\mathcal{E}_{\overline{N}_{\Gamma}}, 
\quad \Bigl[\bigl((q,x); v\bigr)\Bigr] \mapsto (q;[(x,v)]).
$$
One can easily check that this is well-defined, since
$$
I \bigl((q,\gamma\cdot x); \gamma^{-1}\cdot v\bigr) = \bigl(q; (\gamma\cdot x, \gamma^{-1}\cdot v)\bigr) \in (q;[(x,v)]).
$$
By construction we have a commutative diagram as in Figure \ref{iso-pullbacks}.
Since $I$ is acting as identity in local trivializations, it commutes with the natural flat pullback connections and hence we
have a functoriality result
\begin{align*}
b^*_{(2),\mathrm{max}}(\overline{M},f^*\mathcal{E}_{\overline{N}_{\Gamma}})
&=b^*_{(2),\mathrm{max}}(\overline{M},\mathcal{E}_{f^*\overline{N}_{\Gamma}}), \\
\alpha_{*,\mathrm{max}}(\overline{M},f^*\mathcal{E}_{\overline{N}_{\Gamma}})
&=\alpha^*_{*,\mathrm{max}}(\overline{M},\mathcal{E}_{f^*\overline{N}_{\Gamma}}).
\end{align*}

\begin{figure}[h]
\begin{center}
\begin{tikzpicture}

\draw[->] (5.6,0) -- (6.8,0);
\node at (4,0) {$f^{-1}(U_\alpha) \times \ell^2\Gamma$};
\node at (8.5,0) {$f^{-1}(U_\alpha) \times \ell^2\Gamma$};
\node at (6.3,-0.5) {$\textup{id}$};

\draw[->] (5.3,2) -- (7,2);
\node at (4.3,2) {$\mathcal{E}_{f^*\overline{N}_{\Gamma}}$};
\node at (8.2,2) {$f^*\mathcal{E}_{\overline{N}_{\Gamma}}$};
\node at (6.3,2.5) {$I$};

\draw[->] (4,1.5) -- (4,0.5);
\draw[->] (8.2,1.5) -- (8.2,0.5);

\node at (3.3,1) {$\Psi_\alpha$};
\node at (9,1) {$f^*\Phi_\alpha$};

\end{tikzpicture}
\end{center}

\caption{Isomorphism $f^*\mathcal{E}_{\overline{N}_{\Gamma}} \cong \mathcal{E}_{f^*\overline{N}_{\Gamma}}$ in local trivializations}
\label{iso-pullbacks}
\end{figure}

\noindent \textbf{Step 3)} Now we have 
\begin{align*}
b^*_{(2),\mathrm{max}}(\overline{N},\mathcal{E}_{f^*\overline{N}_{\Gamma}})=b_{(2),\mathrm{max}}(\overline{M},\mathcal{E}_{\overline{M}_{\Gamma}}) &= b^*_{(2),\mathrm{max}}(\overline{M}_{\Gamma}), \\
\alpha_{*,\mathrm{max}}(\overline{N},\mathcal{E}_{f^*\overline{N}_{\Gamma}})=\alpha_{*,\mathrm{max}}(\overline{M},\mathcal{E}_{\overline{M}_{\Gamma}}) &= \alpha_{*,\mathrm{max}}(\overline{M}_{\Gamma}),
\end{align*} 
since $\overline{M}_{\Gamma}$ and $f^*\overline{N}_{\Gamma}$ are isomorphic as Galois $\Gamma$-coverings.\\

\noindent \textbf{Step 4)} Since both $\overline{M}$ and $\overline{N}$ are 
oriented, we have 
$$
*\Delta_{k,\mathrm{rel}}=\Delta_{m-k,\mathrm{abs}}*
$$ where $*$ denotes the Hodge star operator,  
$m:=\dim(\overline{M})=\dim(\overline{N})$ and $k\in \{0,...,m\}$. Since 
\begin{align*}
&*: L^2 \Omega^k (\overline{N}, \mathcal{E}_{\overline{N}_{\Gamma}})
\rightarrow L^2 \Omega^{m-k} (\overline{N}, \mathcal{E}_{\overline{N}_{\Gamma}}), \\
&*:L^2 \Omega^k (\overline{M}, \mathcal{E}_{\overline{M}_{\Gamma}})
\rightarrow L^2 \Omega^{m-k} (\overline{M}, \mathcal{E}_{\overline{M}_{\Gamma}}),
\end{align*} 
are  isometries of Hilbert spaces we can conclude that 
$$
\begin{aligned}
b^k_{(2),\min}(\overline{M}_{\Gamma})=b^{m-k}_{(2),\max}(\overline{M}_{\Gamma})=b^{m-k}_{(2),\max}(\overline{N}_{\Gamma})= b^k_{(2),\min}(\overline{N}_{\Gamma})\\
\alpha_{k,\min}(\overline{M}_{\Gamma}) =\alpha_{m-k,\max}(\overline{M}_{\Gamma}) = \alpha_{m-k,\max}(\overline{N}_{\Gamma}) =\alpha_{k,\min}(\overline{N}_{\Gamma}).
\end{aligned}
$$

\end{proof}

\begin{proof}[Proof of Corollary \ref{universal-c}]

This follows immediately by the above proof and Corollary \ref{pullback-u}.

\end{proof}

\end{document}